\documentclass[12pt]{amsart}
\usepackage{graphicx,psfrag,epsfig}
\usepackage{amsmath}
\usepackage[a4paper]{geometry}
\usepackage{amsfonts}
\usepackage{amssymb}

\newtheorem{theorem}{Theorem}
\newtheorem{lemma}{Lemma}[section]
\newtheorem{corollary}[lemma]{Corollary}
\newtheorem{proposition}[lemma]{Proposition}

\newtheorem*{definition}{Definition}
\theoremstyle{definition}

\newtheorem*{remark}{Remark}
\newtheorem*{remarks}{Remarks}

\numberwithin{equation}{section}

\def\de{\overset{\mathrm{def}}{=}}

\def\mi\mathfrak{i}

\def\mm{\mathfrak{m}}

\def\one{{\mathbf{1}}}

\def\bpi{\boldsymbol{\pi}}

\def \Quad{{\quad\quad\quad\quad\quad\quad\quad\quad\quad\quad\quad
\quad\quad\quad\quad\quad\quad\quad\quad\quad\quad\quad\quad\quad
\quad\quad\quad\quad\quad\quad\quad\quad\quad\quad\quad\quad\quad
\quad\quad\quad}}

\def\eps{{\varepsilon}}
\def\teps{{\tilde\varepsilon}}

\def\supp{{\rm supp}}

\def\Card{{\rm Card}}
\def\Const{{\rm Const}}
\def\Corr{{\rm Corr}}
\def\Cov{{\rm Cov}}

\def\Leb{{\rm Leb}}

\def\Prob{{\mathbb{P}}}
\def\tProb{{\tilde{\Prob}}}
\def\Var{{\rm Var}}

\def\EXP{{\mathbb{E}}}
\def\tEXP{{\tilde\EXP}}

\def\bbT{\mathbb{T}}

\def\bbX{\mathbb{X}}

\def\IS{\mathbb{S}}

\def\reals{\mathbb{R}}

\def\integers{\mathbb{Z}}

\def\bE{\mathbf{E}}

\def\bH{\mathbf{H}}

\def\bH{\mathbf{H}}

\def\bP{\mathbf{P}}

\def\bc{\mathbf{c}}

\def\bm{\mathbf{m}}

\def\bu{\mathbf{u}}

\def\brC{{\bar C}}

\def\brbE{{\bar\bE}}

\def\brK{{\bar K}}

\def\brP{{\bar P}}

\def\brbP{{\bar\bP}}

\def\bra{{\bar a}}

\def\brra{{\bar\bra}}

\def\brc{{\bar c}}

\def\brf{{\bar f}}

\def\brg{{\bar g}}

\def\brh{{\bar h}}

\def\brn{{\bar n}}

\def\bru{{\bar{u}}}

\def\brs{{\bar s}}

\def\brt{{\bar t}}

\def\brv{{\bar v}}

\def\brDelta{{\bar \Delta}}

\def\breps{{\bar\eps}}

\def\brpi{{\bar \pi}}

\def\brgamma{{\bar \gamma}}

\def\brlambda{{\bar \lambda}}

\def\brnu{{\bar \nu}}

\def\brmu{{\bar \mu}}

\def\brtheta{{\bar \theta}}

\def\brPhi{{\bar \Phi}}

\def\brzeta{{\bar \zeta}}

\def\brOmega{{\bar\Omega}}

\def\cB{\mathcal{B}}

\def\cC{\mathcal{C}}

\def\cF{\mathcal{F}}

\def\cH{\mathcal{H}}

\def\cO{\mathcal{O}}

\def\cR{\mathcal{R}}

\def\fh{\mathfrak{h}}

\def\hH{{\hat H}}

\def\hP{{\hat P}}

\def\hR{{\hat R}}

\def\ha{{\hat a}}

\def\hc{{\hat c}}

\def\hf{{\hat f}}

\def\hGamma{{\hat\Gamma}}

\def\hnu{{\hat\nu}}

\def\tA{{\tilde A}}

\def\tC{{\tilde C}}

\def\tF{{\tilde F}}

\def\tH{{\tilde H}}

\def\tP{{\tilde P}}

\def\tT{{\tilde T}}

\def\tX{{\tilde X}}

\def\tY{{\tilde Y}}

\def\ta{{\tilde a}}

\def\tc{{\tilde c}}

\def\tf{{\tilde f}}

\def\tn{{\tilde n}}

\def\tp{{\tilde p}}

\def\ts{{\tilde s}}

\def\tu{{\tilde u}}

\def\ty{{\tilde y}}

\def\teps{{\tilde\varepsilon}}

\def\tmu{{\tilde\mu}}

\def\tnu{{\tilde\nu}}

\def\tpi{{\tilde\pi}}

\def\tTheta{{\tilde\Theta}}

\def\tomega{{\tilde{\omega}}}

\def\tPhi{{\tilde{\Phi}}}
\def\ttPhi{{\tilde{\tPhi}}}

\def\tOmega{{\tilde{\Omega}}}

\makeatother

\def\beq{\begin{equation}}
\def\eeq{\end{equation}}

\begin{document}
\title[Subdiffusive random walks on a strip]
{Limit theorems for random walks on a strip\\
in subdiffusive regime.}

\author{D. Dolgopyat$^{1}$ and I. Goldsheid$^{2}$}
\footnotetext[1]{Department of
Mathematics and Institute of Physical Science and Technology,\\
University  of Maryland, College Park, MD, 20742, USA}
\footnotetext[2]{School of Mathematical Sciences, Queen
Mary University of London, London \\
E1 4NS, Great Britain}

\begin{abstract}\noindent  We study the asymptotic behaviour of occupation times of a transient random walk in
quenched random environment on a strip in a  sub-diffusive regime. The asymptotic behaviour of hitting times,
which is a more traditional object of study, is the exactly same. As a particular case, we solve a long
standing problem of describing the asymptotic behaviour of a random walk with bounded jumps on a
one-dimensional lattice. Our technique results from the development of ideas
from our previous work \cite{DG} on the simple random walks in random environment
and those used in \cite{BG1, BG2, G1} for the study of random walks on a strip.


\end{abstract}

\maketitle

%\footnotetext[1]{Department of
%Mathematics and Institute of Physical Science and Technology\\ University  of Maryland,
%College Park, MD, 20742, USA}
%\footnotetext[2]{School of Mathematical Sciences Queen
%Mary, University of London, London \\E1 4NS, Great Britain}

\section{Introduction}
The main goal of this work is to describe the asymptotic behaviour of a random walk (RW)
in a quenched random environment (RE) on a strip in a sub-diffusive regime. As a corollary
we obtain a solution to a long standing problem about the asymptotic behaviour of a RW
with bounded jumps in RE on a one-dimensional lattice.
These two models are natural generalizations of the one-dimensional RWRE with jumps
to the nearest neighbors - the so called simple RWRE (SRWRE). The techniques and ideas
used in this paper resulted from the development and combination of those used in \cite{DG},
where we studied the limiting behaviour of the SRWRE, and in
\cite{BG1, BG2,G1}, which studied of RWRE on a strip. Our main model is the RWRE on a strip
and the main quantitative characteristic of the walk that is the occupation time $T_N$ of a large box
(see \eqref{TN1} for exact definition). In \cite{DG} we also studied $T_N$, but on a strip the
approach we use is very different from the one used for SRWRE.
The important difference between SRWRE and other
models can be roughly explained by the fact that a transient simple walk has to visit every point on its
way to $\infty$, while on a strip it can miss any point with a positive probability.
Due to this fact, the expectations of the occupation times of the sites form a Markov process
in the 'simple' case but this is not true for a walk on a strip. In order to resolve these
difficulties, we have to use methods inspired by the theory of dynamical systems such
as products of random transformations, Lyapunov exponents, transfer operators combined with more probabilistic
techniques such as coupling, large deviations, Poisson processes etc. We believe that
the new point of view presented in this paper makes the proofs more transparent even in the classical SRWRE setting.

We now recall the exact definitions of all three models.

\textbf{Model 1.} In the simplest 1D case, a \textit{random environment} is a  sequence of independent
identically distributed (i.i.d.) random variables $\omega=\{p_n\}_{n\in\mathbb{Z}}$, where $p_n$ are
viewed as probabilities of jumps from $n$ to $n+1$. Given $\omega$ and $X_0=z$, one defines a Markov chain
$X_t,\ t=0,1,...,$ on $\mathbb{Z}$ with a transition kernel given by
\begin{equation}\label{simpleRW}
 \mathbb{P}_\omega(X_{t+1}=k+1|X_t=k)=p_k, \ \ \mathbb{P}_\omega(X_{t+1}=k-1|X_t=k)=1-p_k .
\end{equation}

\textbf{Model 2.} The RWRE on a strip $S\overset{\mathrm{def}}{=}\mathbb{Z} \times\{1,\ldots,m\}$
was introduced in \cite{BG1}
and will be the main object of our study. We say that the set
$L_n\overset{\mathrm{def}}{=}\{(n,j):\,1\le j\le m\}\subset S$ is the \textit{layer} $n$ of
the strip (or just layer $n$). The walker is allowed to jump from a site in $L_n$ only to a site
in $L_{n-1}$, $L_n$, or $L_{n+1}$. Let $X_t=(Z_t,Y_t)$ denote the coordinate of the walk at
time $t$, where $t=0,\,1,\,2,...,$ $Z_t\in \mathbb{Z}$, $1\le Y_t\le m$. An environment $\omega$ on
a strip is a sequence of triples of $m\times m$ matrices $\omega = \{(P_n, Q_n, R_n)\}_{n\in \mathbb{Z}}$
with non-negative matrix elements and such that
$P_n+Q_n+R_n$ is a stochastic matrix:
\begin{equation}\label{stch}
(P_n+Q_n+R_n)\mathbf{1}=\mathbf{1},
\end{equation}
where $\mathbf{1}$ is a vector whose all components are equal to 1.
The transition kernel of the walk is given by
\begin{equation}\label{RWstrip}
\mathbb{P}_\omega(X_{t+1}=z'|X_t=z)=%\cQ(z, z')=
\begin{cases}
P_n(i,j)& \text{ if } z=(n,i), z'= (n+1, j)), \\
Q_n(i,j)&\text{ if } z=(n,i), z'=(n-1, j)), \\
R_n(i,j)&\text{ if } z=(n,i), z'= (n, j))
\end{cases}
\end{equation}
The corresponding Markov chain is completely defined if we set $X(0)=z$.

Throughout the paper we suppose that the following conditions are satisfied:
\begin{equation} \label{EqC1}
\{(P_n, Q_n, R_n)\}_{n\in \mathbb{Z}}  \text{ is an i.i.d. sequence }
\end{equation}
\begin{equation}\label{EqC2*}
\begin{aligned}
&\text{There is an $\varepsilon>0$ such that $\mathbf{P}$-almost surely for all $i,\,j\in[1,m]$}\\
& \left\|  R_n\right\| < 1-\varepsilon,\ \ ((I-R_n)^{-1}P_n)(i,j)>\varepsilon, \ \ ((I-R_n)^{-1}Q_n)(i,j)>\varepsilon.
\end{aligned}
\end{equation}
\begin{remarks}
1. The matrices $P_n$, $Q_n$, and $R_n$ are comprised of probabilities
of jumps from sites in $L_n$ to sites in $L_{n+1}$, $L_{n-1}$, and $L_n$ respectively.
Condition \eqref{stch} is equivalent to `the nearest layer jumps only' property of the walk.

2. Note that $((I-R_{n})^{-1} P_{n})(i,j)$ and $((I-R_n)^{-1}Q_n)(i,j)$ are the probabilities
for a RW starting from $(n,i)$ to reach $(n+1,j)$ and, respectively, $(n-1,j)$ at its first exit from layer $n$.

3. We chose to work under conditions \eqref{EqC2*} in order to simplify the proofs.
In fact all main results can be proved under the following much milder conditions.
\begin{equation}\label{EqC3*}
\begin{aligned}
&\text{There is $\varepsilon>0$ and integer $l\ge 1$ such that $\mathbf{P}$-almost surely $\forall$ $i\in[1,m]$}\\
& \left\|  R_n^l\right\| < 1-\varepsilon,\ \ ((I-R_n)^{-1}P_n)(i,1)>\varepsilon, \ \ ((I-R_n)^{-1}Q_n)(i,1)>\varepsilon.\\
\end{aligned}
\end{equation}
\end{remarks}

Let us describe explicitly the probability spaces hidden behind the above definitions.
By $(\Omega,\mathcal{F},\mathbf{P})$ we denote the probability space describing random environments,
where $\Omega=\{\omega\}$ is the set of all environments,
$\cF$ is the natural sigma-algebra of subsets of $\Omega$
and $\bP$ is a probability measure on $(\Omega,\cF)$. The RWRE is specified by the choice of $\Omega$ and $\bP$. Next,
let $\mathfrak{X}_z=\{X(\cdot): X(0)=z\}$ be the space of all trajectories
of the walk starting from $z\in L_0$.
A \textit{quenched} (fixed) environment $\omega$
thus provides us with a conditional probability measure $\Prob_{\omega,z}$ on
$\mathfrak{X}_z$ with a naturally defined probability space
$(\mathfrak{X}_z, \mathcal{F}_{\mathfrak{X}_z},\Prob_{\omega,z})$. In turn, these two measures
generate a semi-direct product measure $\mathrm{P}_z:=\bP\ltimes\Prob_{\omega,z}$ which is the {\it annealed}
probability measure on $(\Omega\times\mathfrak{X}_z,\mathcal{F}\times \mathcal{F}_{\mathfrak{X}_z})$ .

The expectations with respect to $\Prob_{\omega,z}$, $\mathbf{P}$,
and $\mathrm{P}_z$ will be denoted by $\EXP_{\omega,z}$, $\mathbf{E}$, and $\mathrm{E}_z$ respectively.
%With a slight abuse of notations we may also use $\mathbf{E}$ instead of $\mathrm{E}$.

\begin{remark} The notations $\mathfrak{X}_z$, $\Prob_{\omega,z}$, $\mathrm{E}_z$
etc. emphasize the dependence of these objects
on the starting point $z$ of the walk. However, we often use the simplified version of these
notations such as $\Prob_{\omega}$, $\EXP_{\omega}$, etc. because the asymptotic behaviour
of the walk does not depend on $z$ and it is usually clear from the context what the
starting point of the walk is.
\end{remark}

\textbf{Model 3.} The random walk on $\mathbb{Z}$ with uniformly bounded jumps is
another natural extension of the nearest neighbour model. The random
environment ${\omega}\overset{\mathrm{def}}{=}\{p(x)=(p(x,k))_{-m\leq k\leq
m}\}_{x\in\mathbb{Z}},$ where $p(x)$ is a stationary in $x$ sequence of vectors with
$\sum_{k=-m}^{m}p(x,k)=1$ and $p(x,k)\ge 0$. For a given environment $\omega$ the
transition kernel of the walk is defined by
\begin{equation}\label{wbd}
\mathbb{P}_{\omega}\left(  X(t+1)=x+k\,|\,X(t)=x\right) = p(x,k),\quad x\in\mathbb{Z}%
\end{equation}
The following geometric construction transforms this walk into a walk on a strip.
Let us view $\mathbb{Z}$ as a subset of the $X$-axis in a two-dimensional plane.
Cut the $X$-axis into equal intervals of length $m$ so that each of
them contains exactly $m$ consecutive integer points. Turn each such
interval around its left most integer point anti-clockwise
by $\pi/2$. The image of $\mathbb{Z}$ obtained in this way is a
part of a strip with distances between layers equal to $m$.
Re-scaling the $X$-axis of the plane by $m^{-1}$ makes the
distance between these layers equal to one and the RW on $\mathbb{Z}$
transforms into a RW on a strip with jumps to nearest layers only. The
relevant formulae for matrices $P_n,\ Q_n,\ R_n$ can be found in
\cite{BG1}, where this construction was described in a more
formal way.

It is obvious that if $p(x),\ x\in \mathbb{Z}$, is an i.i.d sequence then the just defined triples
of matrices $(P_n,\ Q_n,\ R_n)$ are i.i.d.
It is also easy to see that \eqref{EqC2*} is satisfied if for some $\varepsilon>0$
\begin{equation}\label{C4}
\mathbf{P}\{ p(x,1)>\varepsilon\,\ p(x,-1)>\varepsilon,\ p(x,m)>\varepsilon,\ p(x,-m)>\varepsilon\} = 1.
\end{equation}
A much wider class of one-dimensional RW with bounded jumps is obtained if instead of
\eqref{C4} we suppose only that
\begin{equation}\label{C5}
\mathbf{P}\{ p(x,1)>\varepsilon,\ p(x,-1)>\varepsilon\} = 1.
\end{equation}
In this case \eqref{EqC2*} may not be satisfied but
\eqref{EqC3*} is satisfied.

\smallskip\textbf{Brief comments on the history of the subject.}
Two pioneering papers which initiated the development of the theory of RWRE were published
in 1975 by Solomon \cite{So} and Kesten, Kozlov, Spitser \cite{KKS}.
In \cite{So} the asymptotic properties of the SRWRE were discussed at the level of the
Law of Large Numbers and the surprising fact that for a wide class of parameters
the SRWRE would be escaping to $\infty$ at a zero speed was discovered.
In \cite{KKS} the \textit{limiting
distributions} of hitting times and of the position of $X$
were found in the \textit{annealed} setting. The extensions of the main results
from these papers to the RWRE on a strip are explained below in Theorems \ref{ThSpeed}, \ref{CLTsg2},
and \ref{ThAnn}.

In 1982, Sinai \cite{S} described the asymptotic behaviour of a recurrent SRWRE.
He discovered a phenomena which is now called the Sinai diffusion.

The methods used in \cite{So, KKS, S} rely heavily on the
\textit{jumps to the nearest neighbours only} property of the walk and
the limiting distributions described in \cite{KKS} were obtained for \textit{annealed} RWRE.
Therefore the following questions arose and were known essentially since 1975:
%(and some of them were explicitly asked later by Sinai in \cite{S}):

1. Can one describe the limiting behaviour of the \textit{quenched} RW at least in the case of the
SRWRE (model 1)?

2. What are the analogues of (a) $\mathrm{P}$-almost sure results from \cite{So}, (b) the  annealed
limiting statements from \cite{KKS} %and (c) the Sinai diffusion
for more general models, saysuch as model 3?

3. What can be said about more general classes of environments, say stationary environments
with appropriate mixing properties?

In the 1982 paper Sinai explicitly stated the questions about the possibility to extend his
results to more general models, such as model 3.

\smallskip The attempts to find answers to question 1 are relatively recent. We shall not discuss them
here in any detail. The references concerned with SRWRE along with relevant discussion can be found
in \cite{G} and \cite{DG}.

Partial answers to question 2 were obtained in \cite{Br, Br1, Der, Ke, L, Let}.
The discussion of these results can be found in \cite{BG1, G1}.

Question 3 was addressed in several publications, see e.g. \cite{WRZ, BG1, G, G1, Z}.
And even though in \cite{DG} and in this work we consider the so called i.i.d environments
(as defined above) we believe that the methods we use are useful for the analysis
of RW in stationary RE satisfying appropriate mixing conditions.

Finally, let us mention several results on the RWRE on a strip
which are directly related to this work. The criterion for
recurrence and transience has been found in \cite{BG1}. A detailed
description of the limiting behaviour in the recurrent regime was given in \cite{BG2}.
A criterion for linear growth and the quenched (and hence annealed) Central Limit Theorem (CLT)
was obtained in \cite{G1} for wide classes of environments; in particular the
CLT for hitting times was established for stationary environments.
%We note that the results of the present work imply that the sufficient conditions
%for diffusive behaviour found in \cite{G1} are in fact also necessary.

\smallskip\noindent
\textbf{Quantities characterizing the asymptotic behaviour of a RWRE}

%\noindent\textbf
%\subsection{}

\smallskip\noindent Remember that $X_t=(Z_t,Y_t)$ is the coordinate of the walk at time $t$ with $Z_t$ being its
$\mathbb{Z}$ component. Denote by $\tT_N$ the hitting time
of layer $L_N$ -- the time at which the walk starting from a site in $L_0$ reaches
$L_N$ for the first time. It is both natural and in the tradition of the field to consider the
understanding of the main asymptotic properties of the walk as achieved
if the asymptotic behaviour of $Z_t$ as $t\to\infty$ and $\tT_N$ as $N\to\infty$ is known.

There is of course a strong connection between the asymptotic behaviour of $Z_{t}$
and $\tT_{N}$. Obviously $\tT_{N}$ is strictly monotone in $N$ and $Z_{\tT_{N}}=N$. This and
some other, less trivial relations between these random variables were used
in a very efficient way in the study of transient RWs already in \cite{So,KKS}.
In particular in \cite{KKS} the asymptotic distribution  of $Z_{t}$ was deduced
from that of $\tT_{N}$.

In our recent work \cite{DG} on SRWRE we studied a different
quantity as the main way of describing the asymptotic behaviour of the RW.
Namely, we considered the \textit{occupation time} $T_N$ of the interval $[0,N-1]$.
The asymptotic behaviour of $\tT_{N}$ is exactly the same as that of $T_N$ since
$|T_N-\tT_{N}|$ is a stochastically bounded random variable (see Lemma 2.1 from \cite{DG}).
%Yet despite such a strong relation between these two r.v.s, our methods don't work for
%$\tT_{N}$ but do for $T_{N}$ and seem to be very efficient.
In this paper, we study a similar quantity - the occupation time of a box
$[L_0,L_{N-1}]\overset{\mathrm{def}}{=}\{(n,i):\, 0\le n\le N-1\} $.
\begin{definition} The {\it occupation time} $T_N $ of the box $[L_0,L_{N-1}]$
is the total time the walk $X_t$ starting from a site in $L_0$ spends on this box
during its life time. In other words
\begin{equation}\label{TN1}
T_N=\#\{t\colon \ 0\le t< \infty,\ X_t\in [L_0,L_{N-1}]\},
\end{equation}
\end{definition}
\begin{remark} Note that $T_N\equiv T_{N,z}$ depends on the starting point $z$ of
the walk. Also, we use the convention that starting from a site $z$
counts as one visit to $z$.
\end{remark}

\noindent
\textbf{The paper is organized as follows.} We start (Section \ref{s1.3}) by reviewing the results
from \cite{BG1, G1} which are used in this paper. In Section \ref{secOccupTmes} we derive formulae
for the expected value of occupation times and state their asymptotic properties;
the latter play a major role in the analysis of the asymptotic behaviour of the RW on a strip.
In Section \ref{MR} we define traps and state the main results of the paper (Theorems \ref{ThTraps}
and \ref{ThMain}) which are followed by Theorem \ref{ThAnn} extending to the case of the strip
the classical results from \cite{KKS}. Section \ref{proofTh5} is devoted to the proof of the
properties of traps followed by the derivation of Theorem \ref{ThTraps}.
The proof of Theorem \ref{ThMain} is given in Section \ref{proofTh6}.
Since this proof is similar to that of the main result in \cite{DG}, we focus our attention
on the differences which are due to the fact that this time we deal with a strip.
Section \ref{ScExt} contains the extensions of our results which are not needed in the analysis of the hitting time but
are important for understanding of other properties of RWRE (cf \cite{ESZ, GS, K2,La} for related work in the context of SRWRE)
and will be used in the future work.
The paper has four appendices containing results which are not specific to
RWRE. Most of these results are not completely new, but we present them in the form convenient for our
purposes. Namely, Appendix \ref{AppOccupation} contains the estimates of occupation times for general transient Markov chains.
Standard facts about the Poisson processes and their relation to stable laws are collected in Appendix \ref{SSPP}.
In Appendix \ref{AppRenewal} we prove a renewal theorem for a system of random contractions.
The fact that the assumptions of Appendix \ref{AppRenewal} are applicable in our setting is verified in Section \ref{tail}.
Appendix \ref{ScIndBC} contains the results about mixing properties of random walks on the strip satisfying ellipticity
conditions.

\noindent
\textbf{Some conventions and notations.}

Letters $C,\ \brC,\ c, \mathbf{c}$ denote positive constants, $\varepsilon$ is a strictly
positive and small enough number, and $\theta$ is a constant from the interval
$(0,1)$. The values of all these constants may be different in different sections of the paper.

$[L_a,L_b]\overset{\mathrm{def}}{=}\{(n,i):\, a\le n\le b\} $ is the part of the strip (a box)
 contained between layers $L_a$ and $L_b$, where $a<b$. We use the notation
 $[a,b]$ and the term interval $[a,b]$ for the box $[L_a,L_b]$ when the meaning of this notation
 is clear from the context.

 $\cF_{a,b}$ is the $\sigma$-algebra of events depending only on the environment in $[L_a, L_b].$

 $e_y$ is a vector whose $y$-th coordinate is $1$ and all others are zeros.

 $\mathbf{1}$ is a column vector with all components equal to 1.

If $x=(x(j))$ is a
vector and $A=(a(i,j))$ a matrix we put
\[
\left\|  x\right\|  \overset{\mathrm{def}}{=}\max_{j}|x(j)%
|\ \hbox{ which implies }\ \left\|  A\right\| =
\max_{i}\sum_{j}|a(i,j)|.
\]

We say that $A$ is strictly positive (and write $A>0)$ if all its
components satisfy $a(i,j)>0$. $A$ is called non-negative (and we
write $A\geq0)$ if all $a(i,j)\ $are non-negative. A similar
convention applies to vectors. We shall make use of the following
easy fact:
\[
\hbox{ if }\ A\ge0 \hbox{ then }\ \left\|A\right\|=
\left\|A\mathbf{1}\right\|.
\]
$\bbX$ denotes the set of non-negative unit vectors,
$\bbX=\{{x}:\, x\in\mathbb{R}^{m},\, x\ge 0, \left\|x\right\|=1\}$.

$E_{\mu}(f)$, $\nu(g)$ denote the expectations of functions $f$ and $g$ over
 measure $\mu$ and $\nu$ respectively defined on the relevant probability spaces.

We often deal with $N^{\eps}$, $\ln N$, $\ln\ln N$, etc which are viewed as integer numbers.
Strictly speaking, we should write $\lfloor N^{\eps}\rfloor$, $\lfloor \ln\ln N\rfloor$, etc.
However, our priority lies with the simpler notation and the exact meaning
is always obvious from the context.

\section{Review of related results from previous work.}
\label{s1.3}
The purpose of this review is to list those results from \cite{BG1} and \cite{G1} which will be
used in this work as well as to put the results of the present work into the right context.
We note that many of the statements listed below were proved in \cite{BG1, G1} under assumptions which are much
milder than \eqref{EqC2*}.

\subsection{Auxiliary sequences of matrices.}
Let us fix $a\in\mathbb{Z}$ and define for $n\geq a$ two sequences
of matrices: $\varphi_n$ and $\psi_{n}$. To this end put
$\varphi_a\overset{\mathrm{def}}{=}0$ and let $\psi_{a}$
be a stochastic matrix. For $n>a$ matrices
$\varphi_n$ and $\psi_{n}$ are defined recursively:
\begin{equation}
\varphi_{n}\overset{\mathrm{def}}{=}(I-R_{n}-Q_{n}\varphi_{n-1})^{-1}P_{n},\
\ %
\psi_{n}\overset{\mathrm{def}}{=}(I-R_{n}-Q_{n}\psi_{n-1})^{-1}P_{n}%
\label{EqPsi}%
\end{equation}
Note that the existence of $(I-R_{n}-Q_{n}\psi_{n-1})^{-1}$ follows
from \eqref{EqC2*}.

%The limiting behaviour of matrices $\varphi_n$ and $\psi_n$ when
%$a\to -\infty$ played a very important role in \cite{BG1, G1} and will
%continue to do so in this paper.

\smallskip\noindent\textbf{ Properties of matrices $\varphi_n$.}
We start with the probabilistic definition of $\varphi_n\equiv\varphi_{n,a}=\left(\varphi_{n,a}(i,j)\right)$
(which implies equation (\ref{EqPsi}) for $\varphi_n$):
\[
\varphi_{n,a}(i,j) =\mathbb{P}_\omega\left(\hbox{RW starting from $(n,i)$  hits $L_{n+1}$
at $(n+1,j)$ before visiting $L_a$} \right).
\]
Obviously these probabilities are monotone functions of $a$ and
hence the limits
$\eta_n\overset{\mathrm{def}}{=}\lim_{a\to -\infty}\varphi_{n,a}$
exist for all (!) environments $\omega$.
Lemma 4 in \cite{BG1} implies that
if \eqref{EqC2*} is satisfied then
$\eta_n>0$ for $\mathbf{P}$-almost every $\omega$
and for $n>a$
\begin{equation}%\label{eps}
\varphi_n(i,j)>\varepsilon,\ \
\ \hbox{ $\psi_n(i,j)>\varepsilon$\ \ for $\mathbf{P}$-almost every $\omega$}.
\end{equation}
%The proof follows from (\ref{EqPsi}): since
%$(I-R_{n}-Q_{n}\varphi_{n-1})^{-1}\ge(I-R_{n})^{-1}$ we have that
%$\varphi_{n}\ge(I-R_{n})^{-1}P_{n}$ and this proves the first
%estimates in (\ref{eps}). The second one is obtained in the same
%way (note that it also follows from $\psi_n\ge\varphi_{n}$).

\smallskip\noindent\textbf{Definition of matrices $\zeta_n$.} It is
easy to see that since $\psi_a$ is stochastic, so are all the
$\psi_{n}$, $n> a$ (Lemma 2 in \cite{BG1}). The following statement
from \cite{BG1} describes the $a\to-\infty$ limits of $\psi_n\equiv
\psi_{n}(\psi_a)$ and defines a stationary sequence of stochastic
matrices $\zeta_n$.

\begin{theorem}
\label{ThZeta} Suppose that Condition \eqref{EqC2*} is satisfied.
Then

(a) For $\mathbf{P}$-a.e. sequence $\omega$ there exists $\zeta_{n}=\lim_{a\rightarrow-\infty}\psi_{n}(\psi_a),$
where the convergence is uniform in $\psi_a$ and the limit $\zeta_{n}$ does not depend on the choice of the sequence $\psi_a$.

(b) The sequence $\zeta_{n}=\zeta_{n}(\omega),\ -\infty<n<\infty,$
of $m\times m$ matrices is the unique sequence of stochastic
matrices which satisfies the following system of equations
\begin{equation}
\zeta_{n}=(I-Q_{n}\zeta_{n-1}-R_{n})^{-1}P_{n},\quad n\in\mathbb{Z}.%
\label{EqZeta}%
\end{equation}

(c) The enlarged sequence
$(P_{n},Q_{n},R_{n},\zeta_{n}),\ -\infty<n<\infty,$
is stationary and ergodic.
\end{theorem}
\begin{remark} Statements (a) and (b) imply that
$\zeta_{n}\equiv \zeta_{n}(\omega)$ depends only on the "past" of the environment, namely
on $\omega_{\le
n}\overset{\mathrm{def}}{=}((P_{k},Q_{k},R_{k}))_{k\le n}$.
\end{remark}

We need the following corollary of Theorem \ref{ThZeta} (Remark 4 in \cite{BG1}).
\begin{corollary} \label{Corol1}Suppose that $(P,Q,R)$ satisfies Condition \eqref{EqC2*} (this can be
any triple of matrices from the support of the distribution of $(P_{0},Q_{0},R_{0})$). Then
there is a unique stochastic matrix $\zeta$ such that
\begin{equation}
\zeta=(I-Q\zeta-R)^{-1}P.%
\label{EqZeta1}%
\end{equation}
\end{corollary}
\begin{proof} Consider the environment with transition probabilities which do not change from layer
to layer and are given by matrices $(P,Q,R)$, that is $\omega=\{(P,Q,R)\}$. Then for this single
environment all conditions of Theorem \ref{ThZeta} are satisfied.
Now statement (a) implies that $\zeta_n=\zeta_{n-1}$
and setting $\zeta:=\zeta_n=\zeta_{n-1}$ turns equation \eqref{EqZeta} into \eqref{EqZeta1}.
\end{proof}

\smallskip\noindent
\textbf{The non-arithmenticity condition.}
We are now in a position to introduce the so called non-arithmeticity condition which will
be often used in the sequel. Let $(P,Q,R)$ and $\zeta$ be as in Corollary \ref{Corol1}. Set
\begin{equation}
\label{APer}
A_{(P,Q,R)}=(1-Q\zeta-R)^{-1} Q
\end{equation}
and let $e^{\lambda_{(P,Q,R)}}$ be the leading eigenvalue of $A_{(P,Q,R)}.$
We say that the environment satisfies the {\it non-arithmeticity condition} if
\begin{equation}\label{NA}
\text{  the distribution of $\lambda_{(P,Q,R)}$ is non-arithmetic.}
\end{equation}

\smallskip\noindent
\textbf{Vectors $\pi_n$.} Our sequence of stochastic matrices $\zeta_n$  is such that
$\zeta_n(i,j)\ge\varepsilon>$ for some $\varepsilon>0$. Due to that we can always construct a sequence
$\pi_n$ of probability vectors such that $\pi_n=\pi_{n-1}\zeta_{n-1}$.
Namely, set $\pi_{n,a}=\tpi_a\zeta_a\ldots\zeta_{n-1}$,
where $\tpi_a$ is a probability vector.
\begin{lemma} %(\cite{G1})
\label{lemma1}
If $\zeta_n(i,j)\ge\varepsilon$ for some $\varepsilon>0$ then the following limit exists
and does not depend on the choice of the sequence of probability vectors $\tpi_a$ :
\begin{equation}\label{y}
\pi_n\overset{\mathrm{def}}{=}\lim_{a\to-\infty}
\tpi_a\zeta_a\ldots\zeta_{n-1},
\end{equation}
 Moreover, for $\theta=1-m\varepsilon$
\begin{equation}\label{y1}
||\pi_n-\pi_{n,a}||\le \theta^{n-1-a}\text{ and }\pi_n(i)>\varepsilon\ \hbox{ for any }\ i\in [1, m].
%\prod_{k=a}^{n-1}(1-\epsilon_{k}),
\end{equation}
%where $\epsilon_k\overset{\mathrm{def}}{=}m\,\min_{i,j}\zeta_k(i,j)$.
\end{lemma}
\begin{remarks} 1. In our case vectors $\pi_n\equiv \pi(\omega_{\le
n})$ form a stationary sequence. \newline
2. Lemma \ref{lemma1} is a well known fact which follows from the usual contracting
properties of products of stochastic matrices. We state it here for future references.
\end{remarks}

\noindent \textbf{Matrices $A_n$ and Lyapunov exponents.}
We can finally define the following sequence of matrices:
\begin{equation}
A_{n}\overset{\mathrm{def}}{=}(I-Q_{n}\zeta_{n-1}-R_{n})^{-1}Q_{n}.%
\label{DefA}%
\end{equation}

Obviously, $A_{n}$ is a stationary sequence and the top Lyapunov
exponent of the product of matrices $A_{n}$ is defined as usual by
\begin{equation}
\lambda\overset{\mathrm{def}}{=}\lim_{n\rightarrow\infty}{\frac{1}{{n}}%
}\log\left\|  A_{n}A_{n-1}\dots A_{1}\right\|.\label{DefLyap}%
\end{equation}
It is well known (see \cite{FK}) that with $\mathbf{P}$-probability 1
the limit in (\ref{DefLyap}) exists and does not depend on
$\omega$.

\subsection{Recurrence and transience of RWRE} The recurrence criteria was proved in \cite{BG1}
for a RWRE on a strip in very general ergodic setting. We need the following particular case of this
 result.
\begin{theorem}
\label{ThLyap} Suppose that Conditions \eqref{EqC1} and \eqref{EqC2*} are satisfied.
Then the following statements hold for $\mathbf{P}$-a.e. $\omega$, $\mathbb{P}_{\omega}$-almost surely:
\newline
(a) $\lambda<0$ iff the RW is transient to the right: $\lim_{t\rightarrow\infty}X_t=\infty ,$
\newline
(b) $\lambda>0$ iff
the RW is transient to the left: $\lim_{t\rightarrow\infty}X_t=-\infty$,
\newline
(c) $\lambda=0$ iff the RW is recurrent:  %$\mathrm{P}$-almost surely
$ \limsup_{t\rightarrow\infty}X_t=+\infty$ and $\liminf_{t\rightarrow\infty}
X_t=-\infty$.
\end{theorem}
\begin{remark} %Remember that "$\mathrm{P}$-almost surely" is equivalent to
%"$\mathbb{P}_{\omega}$-almost surely for $\mathbf{P}$-a.e. $\omega$".
The proof of Theorem \ref{ThLyap} given in
\cite{BG1} contains the following useful statement:
the RW is recurrent or transient to the right if and only iff
\begin{equation}
\lim_{a\to-\infty}\varphi_k=\zeta_k
\label{forLemma7}%
\end{equation}
\end{remark}

\subsection{Moment Lyapunov exponents $r(\alpha)$.}
From now on we consider RWRE which are transient to the
right, that is $\lambda <0$. Let us define a function whose properties are
responsible for the speed of growth of our RW.

Let $A_n$ be a sequence of matrices defined by (\ref{DefA}). For $\alpha\ge 0$ put
\begin{equation}
r(\alpha)\overset{\mathrm{def}}{=}
\limsup_{n\to\infty}\left(\mathbf{E}||A_{n} \cdots
A_{1}||^\alpha\right)^{\frac{1}{n}}.
\label{r}%
\end{equation}
Note that if $m=1$ then $\zeta_n=1$, $A_{n}=q_n/p_n$, and $r(\alpha)=\mathbf{E}(q_0/p_0)^{\alpha}$.
In this form $r(\alpha)$ was first introduced in \cite{KKS}.

\begin{lemma}
\label{Mainlemma}
Suppose that \eqref{EqC2*} is satisfied.
Then:
\newline
(a) the following limit exists and is finite for every
$\alpha\ge0$:
\begin{equation} \label{r2}
r(\alpha)=\lim_{n\to\infty}\left(\mathbf{E}||A_{n} \cdots
A_{1}||^\alpha\right)^{\frac{1}{n}}.
\end{equation}
\newline
(b) the convergence in (\ref{r2}) is uniform in
$\alpha\in[0,\,\alpha_0]$ for any $\alpha_0>0$
\newline
(c) $r'(0)=\lambda$.
%\newline
%(d) there is a strictly positive and continuous function
%$\mathrm{b}(\alpha, x)$ defined on $[0,\infty)\times \mathbb{X}$
%such that
%\begin{equation}
%\frac{\mathbf{E}||A_{n} \cdots
%A_{1}x||^\alpha}{r^n(\alpha)}=\mathrm{b}(\alpha,x)
%+\delta_n(\alpha,x)
%\label{egv}%
%&\end{equation}
%where the sequence $\delta_n(\alpha,x)\to0$ as $n\to\infty$
%uniformly in $(\alpha, x)\in [0,\alpha_0]\times \mathbb{X}$ for
%any $\alpha_0>0$.
\end{lemma}

\subsection{Linear and sub-linear growth of the random walk.}
Let as in the Introduction
$X_t=(Z_t,Y_t)$ be a random walk  starting from a site $z\in L_0$,
$\tT_n$ be the hitting time of layer $L_n$ by this walk.
%and $X(t)$
%be the $\mathbb{Z}$ coordinate of this walk at time $t$.
%the environment $\omega$ is supposed to be fixed.
\begin{theorem}
\label{ThSpeed} Suppose that \eqref{EqC2*}
is satisfied and that $\lambda<0$. Then:

\noindent (i) $r(1)<1$  implies that for $\mathbf{P}$-a.e.
environment $\omega$ with $\mathbb{P}_{\omega,z}$-probability~1
\begin{equation}\label{2.5}
\lim_{n\rightarrow\infty}{n}^{-1}(\tT_n-\mathbb{E}_{\omega,z}\tT_n)=0
\quad \text{and} \quad \lim_{n\rightarrow\infty}{n}^{-1}\mathbb{E}_{\omega,z}\tT_n=c>0
\end{equation}
and for $\mathbf{P}$-a.e. environment the limit $\lim_{n\rightarrow\infty}{n}^{-1}
\mathbb{E}_{\omega,z} (\tT_n) $ exists and is finite.
%\end{equation}
%where
%\begin{equation}\label{2.6}
%\mu=\mathbb{E}\left[y_0(u_0+A_0u_{-1}+\ldots+A_0A_{-1}\ldots
%A_{-k}u_{-k-1}+\ldots)\right],
%\end{equation}
%where the row-vector $y_0$ is defined by (\ref{y}) and the column
%vector $u_k=(I-Q_k\zeta_{k-1}-R_k)^{-1}\mathbf{1}$.

\noindent (ii) $r(1)\geq 1$ implies that for $\mathbf{P}$-a.e.
environment $\omega$ with $\mathbb{P}_{\omega,z}$-probability~1
\begin{equation}\label{2.7}
\lim_{n\rightarrow\infty}{n}^{-1} \tT_n=\infty.
\end{equation}
\end{theorem}
The fact that $\lim_{t\rightarrow\infty}{t}^{-1}Z_t = c^{-1}$ as when $r(1)<1$
and that $\lim_{t\rightarrow\infty}{t}^{-1}Z_t =0$ as when $r(1)\ge 1$
follows from \eqref{2.5} and \eqref{2.7} respectively.

These results extend the relevant statements from \cite{So} to the case of the strip.
Further details can be found in \cite{G1}.

\subsection{The diffusive regime (Central Limit Theorem) for the random walk.}

\begin{theorem}
\label{CLTsg2}
Assume that \eqref{EqC2*} is satisfied, $\lambda<0$, and $r(2)<1.$ Then
there exists $D_1$ such that for $\mathbf{P}$-almost every environment
$ \frac{\tT_n-\mathbb{E}_{\omega, z} \tT_n}{\sqrt{n}} $ converges weakly
as $n\to\infty$
to a normal distribution with zero mean and variance $D_1.$
\end{theorem}
\begin{remark} It is easy to show that if in addition to the conditions of
Theorem \ref{CLTsg2} also \eqref{EqC1}
is satisfied, then there are constants $\bc$ and $D_2$ such that
$\frac{\mathbb{E}_{\omega, z} \tT_n-\bc n}{\sqrt{n}}$ converges weakly
to a normal distribution with zero mean and variance $D_2.$
Consequently $\frac{\tT_n-\bc n}{\sqrt{n}}$ converges weakly as $n\to\infty$
in the annealed setting (that is with respect to $\mathrm{P}_z:=\bP\ltimes\Prob_{\omega,z}$)
to a normal distribution with zero mean and variance $D=D_1+D_2.$
\end{remark}

\section{Occupation times.}\label{secOccupTmes}
As stated in the Introduction, in this work the study the asymptotic behaviour of the RWRE is conducted
in terms of that of the asymptotic behaviour of occupation times. In this section we derive formulae
for the expectations of occupation times and discuss some of their properties.
Denote the time spent by the walk at site $x=(n,y)$ by $\xi_x$. Obviously, the
distribution of $\xi_x$ depends on the starting point of the walk, say $(k,i)$.
Since $\xi_x$ conditioned on the walk starting from $x$ has a geometric distribution, it
is easy to find the parameters of $\xi_x$ also for the walk starting from arbitrary $(k,i)$.
Namely, we shall find $F_{k,i}\de\mathbb{E}_{\omega,(k,i)} (\xi_x)$
for all $(k,i)$ in the strip including $F_{x}=\mathbb{E}_{\omega,x} (\xi_x)$. Then
$$
\mathbb{P}_{\omega,(k,i)}\{X\text{ reaches } x\}=\mathbb{P}_{\omega,(k,i)}\{\xi_x\ge 1\}=
F_{k,i}F_{x}^{-1}%\equiv \mathbb{E}_{\omega,(k,i)} (\xi_x)/\mathbb{E}_{\omega,x} (\xi_x).
$$
and hence
$$
\mathbb{P}_{\omega,(k,i)}\{\xi_x=0\}= 1-F_{k,i}F_{x}^{-1}, \quad
\mathbb{P}_{\omega,(k,i)}\{\xi_x=j\}=F_{k,i}F_{x}^{-2}(1-F_{x}^{-1})^{j-1} \quad \text{if}\quad j\ge1.
$$
The expressions for $F_{k,i}$ will be given in terms of the matrices defined in section \ref{s1.3}.
Denote by $F_k$ the $m$-dimensional
vector with components $F_{k,i},\ 1\le i\le m$ and let $e_y\in \mathbb{R}^m$ be a vector
whose $y^{\mathrm{th}}$ coordinate is 1 and all others are zeros.
\begin{lemma}\label{ExpectedOT} Suppose that \eqref{EqC2*} holds, $x=(n,y)$.
Then for $\mathbf{P}$ - almost all $\omega$
\begin{equation}\label{ExpectF}
 F_k\equiv F_k(n,y)=\sum_{j=n}^{\infty} \zeta_k\dots \zeta_{j-1} A_j \dots A_{n+1} u_{n,y},\quad\text{if }k<n
\end{equation}
\begin{equation}\label{ExpectFgen}
 F_k\equiv F_k(n,y)=\sum_{j=k}^{\infty} \zeta_{k}\dots \zeta_{j-1} A_j \dots A_{n+1} u_{n,y},\quad\text{if }k\ge n
\end{equation}
where
\begin{equation}\label{ExpectF1}
u_{n,y}=(I-Q_n \zeta_{n-1}-R_n)^{-1} e_{y}.
\end{equation}
\end{lemma}
\begin{remark} In the above formulae, we use the conventions that $ A_j \dots A_{n+1}=I$ if $j\le n$
 and $\zeta_{k}\dots \zeta_{j-1}=I$ if $k\ge j$. Thus the first term in \eqref{ExpectF} is
 $\zeta_k\dots \zeta_{n-1} u_{n,y}$ and the first term in \eqref{ExpectFgen} is $A_k \dots A_{n+1} u_{n,y}$.
\end{remark}
\begin{proof} Consider a box $[L_a,L_b]$ with $a<k, n<b$ and a walk $\tX$ on this box starting from
$(k,i)$ with absorbtion at layers $L_a$ and $L_b$. Denote by $\tF_k\de (\tF_{k,i})_{1\le i\le m}$ the
vector of conditional expectations, $\tF_{k,i}\de\mathbb{E}_{\omega,(k,i)} (\xi_x^{a,b})$, where
$\xi_x^{a,b}$ is the occupation time of $x$ by the walk starting from $(k,i)$.
It is easy to see (first step analysis) that
$\tF_k$ satisfy the following system of equations:
\begin{equation}\label{ExpectF2}
 \tF_k=\Phi_k+P_k \tF_{k+1}+R_k \tF_k+Q_k \tF_{k-1}\hbox{ if } a<k<b \hbox{ and }
\tF_a=\tF_b=0,
\end{equation}
where $\Phi_k=e_y$ if $k=n$ and $\Phi_k=0$ otherwise. Systems of equations
of this form were studied in \cite{G1}.
The idea is to look for solutions to \eqref{ExpectF2} of the form $\tF_k=\varphi_{k}\tF_{k+1}+d_k$
with $\varphi_{a}=0$ and $d_a=0$.
A simple calculation (see Lemma 8 in \cite{G1}) shows that $\varphi_{k}$ satisfy \eqref{EqPsi} and that
\begin{equation}\label{ExptF}\tF_k=\varphi_{k} \dots \varphi_{n-1} d_n+
\varphi_{k} \dots \varphi_{n} d_{n+1}+\dots+
\varphi_{k} \dots \varphi_{b-2} d_{b-1}.
\end{equation}
where
$$ d_r=\tu_r+\tA_r \tu_{r-1}+\tA_r \tA_{r-1} \tu_{r-2}+\dots+
\tA_r\dots \tA_j \tu_{j-1}+\dots $$
and
$$ \tu_r=(I-Q_r \varphi_{r-1}-R_r)^{-1} \Phi_r,\ \ \tA_r\de
(I-Q_r \varphi_{r-1}-R_r)^{-1}Q_r. $$
In our case
$d_l=\tA_l \dots \tA_{n+1} \tu_n$ if $l>n,$
$d_n=\tu_n$ and $d_l=0$ otherwise which turns \eqref{ExptF} into a version of \eqref{ExpectF}, \eqref{ExpectFgen}
with $\zeta$'s replaced by $\varphi$'s and the sums being finite.
Note that: \newline (a) $\lim_{a\to-\infty}\varphi_{j}=\zeta_j$ since $\lambda<0$ (see \eqref{forLemma7} or \cite{BG1});
\newline (b) moreover $\varphi_{j}\nearrow \zeta_j$ and therefore also $\tA_j\nearrow A_j$ as $a\to-\infty$;
\newline (c) thus $\tF_k$ is monotonically increasing as $a$ decreases or $b$ increases with the terms
of $\tF_k$ converging to the corresponding terms of $F_k$.
\newline Series \eqref{ExpectF}, \eqref{ExpectFgen} converge $\mathbf{P}$ - almost surely (once again due to $\lambda<0$)
and since they have positive entries, it follows that $\tF_k$ converges $\mathbf{P}$ - almost surely
to  $F_k$ as $a\to-\infty$ and $b\to+\infty$ and this proves the Lemma.
\end{proof}
One immediate corollary from \eqref{ExpectFgen} is the estimate of the probability of return from
$L_{n+l} $ to $L_{n}$. For a $\theta_0<1$, set
\begin{equation}\label{returnOmega}
\Omega_{n, l, \theta_0}=\left\{\omega: \mathbb{P}_{\omega}\{X\text{ visits } L_n \text{ after }L_{n+l}\}\ge\theta_0^l\right\}.
\end{equation}
\begin{lemma}\label{return}
There are $C>0$, $\theta_0, \theta_1\in(0,1)$ such that
\begin{equation}
\bP\left\{\Omega_{n, l, \theta_0}\right\}<C\theta_1^l.
\end{equation}
\end{lemma}
\begin{proof}
It follows from \eqref{ExpectFgen} that $||F_{n+l}(n,y)||\le \Const\sum_{j=n+l}^{\infty} || A_j \dots A_{n+1}||.$
Let $\xi^{n,l}$ be the number of visits to $L_n$  by the walk starting from $L_{n+l}$.
Then
$$
\begin{aligned}
&\mathbb{P}_{\omega}\{X\text{ visits } L_n \text{ after }L_{n+l}\}\le \max_{z\in L_{n+l}}\mathbb{P}_{\omega,z}\{\xi^{n,l}\ge 1\}
\le \max_{z\in L_{n+l}}\mathbb{E}_{\omega,z}(\xi^{n,l})\\
&\le \max_{y} ||F_{n+l}(n,y)||\le \Const\sum_{j=n+l}^{\infty} || A_j \dots A_{n+1}||
\end{aligned}
$$
Fix any $\alpha$ such that $r(\alpha)<1$. Then for any $\theta_0>0$
$$
\bP\{\sum_{j=n+l}^{\infty} || A_j \dots A_{n+1}||\ge \theta_0^l\}\le
\theta_0^{-\alpha l}\sum_{j=n+l}^{\infty} \bE(|| A_j \dots A_{n+1}||)^{\alpha}
\le \Const\theta_0^{-\alpha l} r(\alpha)^l.
$$
and it remains to choose  $\theta_0$ so that $\theta_1:=r(\alpha)\theta_0^{-\alpha}<1$.
\end{proof}
Let us now discuss a corollary which follows from \eqref{ExpectF}. This formula will
quite often be used when $n-k>cN^\eps$ for some $\eps>0$ and $N\to\infty$ in which case
the expression for $F_k(n,y)$ can be simplified. Namely, by Lemma \ref{lemma1}
\begin{equation}\label{prodZeta}
\zeta_k\dots \zeta_{j-1}=\bpi_j +O(\theta^{j-k}),
\end{equation}
where $\bpi_j$ is a rank one matrix all rows of which are equal to $\pi_j.$ We can now rewrite \eqref{ExpectF} as
\begin{equation}\label{ExpectF21}
 F_k(n,y)=\sum_{j=n}^{\infty} (\bpi_{j} +O(\theta^{j-k})) A_j \dots A_{n+1} u_n=
 (1+O(\theta^{n-k}))\sum_{j=n}^{\infty} \bpi_j  A_j \dots A_{n+1} u_{n,y}.
\end{equation}
Fix $\eps>0$ and $\theta\in(0,1).$ Then it is easy to see that for sufficiently large $N$
$$\mathbf{P}\left\{\theta^{N^\eps} \max_{n\in[0,N]}\sum_{j=n}^{\infty}
|| A_j \dots A_{n+1} ||>N^{-50(1+s)}\right\}< N^{-100}.$$
Since errors of such order as well as events of such small probability will not
play any role for our results and in our proofs (see the statements of the
main theorems below), the dependence of $F_k(n,y)$ on the starting point of
the walk can be neglected this vector can be replaced by a single number
\begin{equation}\label{ExpectFy}
 \rho_{n,y}:=\sum_{j=n}^{\infty} \pi_j A_j \dots A_{n+1} u_n.
\end{equation}
We also set
\begin{equation}\label{ExpectF10}
 \rho_{n}:=\sum_y \rho_{n,y}= \sum_{j=n}^{\infty} \pi_j A_j \dots A_{n+1}\bu_n
\end{equation}
where $ \bu_n=(I-Q_n \zeta_{n-1}-R_n)^{-1} \one $. Obviously, $\rho_n$
is the expectation of the time spent by the walk in layer $n$.

Since $A_j$ is a stationary sequence of positive matrices there exists a stationary sequence
of vectors $v_j>0$, $v_j\in\bbX$ and numbers  $\lambda_j>0$ such that
\begin{equation}
\label{DefV}
A_j v_{j-1}=\lambda_j v_{j}
\end{equation}
and there is a sequence of functionals
$l_j$ and a $\theta<1$ such that for any vector $u$
$$ A_j \dots A_{n+1} u=l_n(u) \lambda_j \dots \lambda_{n+1}(1+O(\theta^{n-j})) v_{j}. $$
Moreover $l_n$ is well approximated by local functions, that is,
there exists $\cF_{n, n+r}$ measurable functions $l_{n,r}$ %depending only on $\{P_k, Q_k, R_k\}_{k\in [n, n+r]}$
such that
$||l_n-l_{n,r}||\leq \theta^r $ whenever $r>r_0$ ($r_0$ depends only on $\theta$ and $\varepsilon$
from Condition \eqref{EqC2*}) .

The foregoing discussion allows us to write
\begin{equation}
\rho_{n}=l_n(\bu_n) w_n
+\cR_{n}
\end{equation}
where
\begin{equation}
\label{DefW}
w_n=\sum_{j\ge n} \lambda_j \dots \lambda_{n+1} (\pi_j, v_{j})
\end{equation}
and $\cR_{n}$ denotes the contribution of subleading terms. Denote
$$w_{n_1, n_2}=\sum_{j=n_1}^{n_2} \lambda_j \dots \lambda_{n_1+1} (\pi_j, v_{j}) .$$
The following lemma characterizes the tail behaviour of the distribution of $w_n$ and thus also
of $\rho_n(y)$ and $\rho_n$. It adjusts to our needs some well known results from \cite{K}; the latter
played a major role in many previous studies of the asymptotic behaviour of the RWRE,
in particular in \cite{KKS, DG}. The derivation of this lemma will be given in Section \ref{tail} using the
results of Appendix \ref{AppRenewal}.
The lemma relies on the assumption that
\begin{equation}\label{defs}
\text{there is $s>0$ such that $r(s)=1.$}
\end{equation}
Note that $\ln (r(\cdot))$ is a strictly convex function (see e.g. section \ref{SSLD}) and therefore the existence of the solution
$s$ to \eqref{defs} implies its uniqueness. On the other hand if \eqref{defs} has no positive solutions then
$r(\alpha)<1$ for all $\alpha>0.$ In particular, the walk is diffusive in that case in view of Theorem \ref{CLTsg2}.

For the rest of the paper we suppose that \eqref{defs} is satisfied.

Further analysis will heavily rely on the asymptotic properties of the of tails of distributions of $\rho_n$
which will follow from those of $w_n$. The latter are described by the following lemma.
\begin{lemma}
\label{LmWRen}
Suppose that \eqref{EqC2*} is satisfied.
Then there are constants $\brC$ and $\brs>s$ such that

(a) If $n_2-n_1>\brC \ln t$ then
$$ \bP(w_{n_1}-w_{n_1, n_2}\geq 1|\zeta_{n_1}=\zeta, \pi_{n_1}=\pi, v_{n_1}=v)\leq C t^{-\brs}.$$

(b) If $n_2-n_1>\brC \ln t$ and the non-arithmeticity condition \eqref{NA} holds (in addition to \eqref{EqC2*}) then
there is a function $f(\zeta, \pi, v)>0$ such that
$$ \bP(w_{n_1, n_2}\geq t|\zeta_{n_1}=\zeta, \pi_{n_1}=\pi, v_{n_1}=v)\sim f(\zeta, \pi, v) t^{-s} .$$
In particular
$$ \bP(w_{n}\geq t|\zeta_{n}=\zeta, \pi_n=\pi, v_{n}=v)
\sim f(\zeta, \pi, v) t^{-s} .$$

(c) %For any $k\leq n$ we have
%$$ \bP(\cR_{n}>t|\zeta_{n}=\zeta, \pi_n=\pi, v_{n}=v)\leq C t^{-s}. $$
There exists $\brs>s$ and $\brC>0$ such that %for any $k\leq n-\brC \ln t$ we have
$$ \bP(\cR_{n}>t|\zeta_{n}=\zeta, \pi_n=\pi, v_{n}=v)
\leq C t^{-\brs}. $$
\end{lemma}
The proof of Lemma \ref{LmWRen} is given in Section \ref{tail}.
\section{Main Results\label{MR}.}
%In view of Theorem \ref{CLTsg2} in this paper we assume that
%$r(2)>1$ and hence there exists a unique $s\in (0,2)$ such that $r(s)=1.$

The description of the asymptotic behaviour of $T_N$ (defined by \eqref{TN1}) will be derived from the
asymptotic properties of traps. Our first main result describes these properties.
Let us introduce the exact definition of a trap.

\begin{definition} Let $M=M_N:=\ln\ln N$ and $\delta>0$ be a given (small) number, $w_n$ is defined by
\eqref{ExpectF10}. We say that $n$ is a {\it massive} site if $w_{n}\geq \delta
N^{1/s}$. A site $n\in[0,N-1]$ is {\it marked} if it is massive
and $w_{n+j}<\delta N^{1/s} $ for
$1\leq j\leq M.$
For $n$ marked the interval $[n-M, n]$ is called the {\it trap} (or
$\delta N^\frac{1}{s}$-trap) associated to $n.$
We call the number
$\mm_n=\sum_{j={n-M}}^n \rho_j $
the {\it mass} of the trap.
\end{definition}
Note that this definition implies that distinct traps are disjoint.

%Denote by $\tau:\Omega\mapsto\Omega$ the
%shift operator acting on $\Omega$ in a standard way: if $\omega=\{g_j\}_{j\in \mathbb{Z}}$, where $g_j:=(P_j,Q_j,R_j)$,
%then $\tau(\omega)=\{g_{j+1}\}_{j\in \mathbb{Z}}$.
The asymptotic distribution of traps is described by the following
\begin{theorem} \label{ThTraps}
 Assume that the non-arithmeticity condition \eqref{NA} holds. Then there exists
%a probability measure $\tnu$ on $\Omega$ and
a constant $\bc$ such that the point process
\begin{equation}\label{Poisson}
\left\{ \left(\frac{{n_j}}{N},\,\frac{\mm_{n_j}}{N^{1/s}}\right): %,\, \omega^{(j)}\right):
\ n_j \text{ is } \delta N^{1/s}\text{-marked and}\quad 0\leq n_j\leq N%,\ \omega^{(j)}=\tau^{n_j} \omega
\right\}
\end{equation}
converges as $N\to\infty$ to a Poisson process on $[0,1]\times [\delta, \infty)$%\times \Omega$
with the measure %$\bc s dt'\frac{dt}{t^{1+s}}d\tnu.$
$\bc s dt'\mu_\delta(dt)$, where $\mu_\delta$ converges to a measure with
density $\frac{\bc s}{t^{s+1}}$ as $\delta\to 0$.
\end{theorem}
\begin{remarks} 1. Each component of the point process \eqref{Poisson} is
itself a point process converging to a Poisson process.

2. Theorem \ref{ThTraps} extends
certain statements from \cite{DG} (see Lemma 4.4 there) to the case of the walk on a strip.
It may be worth mentioning that in \cite{DG} we used the term \textit{cluster} for what we have
now decide to call a \textit{trap};
the latter term seems to better reflect the most essential properties of this object.
%However, the existence of $\tnu$ is a new statement even in the case of the SRWRE.

3. The measure $\mu_\delta$ will be described in more explicit terms later. However,
its explicit description depends on the choice of the definition of a trap and
is only important because it helps to find the limit of $\mu_\delta$
as $\delta\to 0$.
\end{remarks}

The above Theorem plays a major role in the description of the asymptotic behaviour
of the walk in the subdiffusive regime $s\in(0,2)$.
To state our second main result, we define $ \mathfrak{t}_N$ which is a
normalized version of $T_N$, namely we set
\begin{equation}\label{tN}
\mathfrak{t}_N=\begin{cases} \frac{T_N}{N^{1/s}} & \text{if } 0<s<1, \cr
\frac{T_N-\EXP_\omega(T_N)}{N^{1/s}} & \text{if } 1\le s<2. \cr
%\frac{T_N-\EXP_\omega(T_N)}{\sqrt{N\ln N}} & \text{if } s=2. \cr
\end{cases}
\end{equation}
The definition of $ \mathfrak{t}_N$ implies that complete understanding of its asymptotic
properties should include the study of those of $\mathbb{E}_{\omega}(T_N)$.
The corresponding normalized quantity is defined as follows:
\begin{equation} \label{ExpectationT}
\mathfrak{u}_N=\begin{cases} \frac{\EXP_\omega(T_N)}{N^{1/s}} & \text{if } 0<s<1, \cr
\frac{\EXP_\omega(T_N)-u_N}{N} & \text{if } s=1, \cr
\frac{\EXP_\omega(T_N)-\mathrm{E}(T_N)}{N^{1/s}} & \text{if } 1< s<2, \cr
%\frac{\EXP_\omega(T_N)-\bE(T_N)}{\sqrt{N\ln N}} & \text{if } s=2, \cr
\end{cases}
\end{equation}
where $u_N=N\mathbf{E}(\rho_n I_{\rho_n<x_N})$ with $x_N$ defined by
$\mathbf{P}\left\{\rho_n>x_N\right\}=N^{-1}$.

To state our next theorem, we set
\begin{equation}\label{Poisson1}\Theta^{(N,\delta)}=
\left\{ \Theta^{(N,\delta)}_j:=\frac{\mm_{n_j}}{N^{1/s}}\,:
0\leq n_j\leq N\quad \text{and } n_j \text{ is } \delta N^{1/s}\text{-marked}\right\}
\end{equation}
\begin{theorem}
\label{ThMain}
Assume that the non-arithmeticity condition \eqref{NA} holds. Then for $0<s<2$ and a $\delta>0$ there is a sequence
$\Omega_{N,\delta}\subset\Omega$ such that
$\lim_{N\to\infty}\bP(\Omega_{N,\delta})=1$ and a sequence of random point processes
$$(\Theta^{(N,\delta)},\Gamma^{(N,\delta)})= \left(\left\{\Theta^{(N,\delta)}_j,\Gamma^{(N,\delta)}_j\right\}\right)$$
%\Omega\times\mathfrak{X}\mapsto \mathfrak{F}_\delta\times \tilde{\mathfrak{F}},$$
such that

%\noindent(i) The component $\Theta^{(N,\delta)}$ depends only on $\omega$ and converges weakly to
%a point Poisson process $\Theta^{(\delta)}$ on $[\delta,\infty)$ with intensity
%$\frac{\brc}{\theta^{1+s}}$ (with some constant $\brc>0$).

\noindent (i) The component $\Gamma^{(N,\delta)}=\{\Gamma^{(N,\delta)}_j\}$ is a collection
of asymptotically i.i.d. random variables with mean 1 exponential distribution.
% which also are asymptotically independent of the environment.

\noindent (ii) The $\mathfrak{t}_N$ and $\mathfrak{u}_N$ can be presented in the following form:

\noindent (a) If\, $0<s<1$ then for $\omega\in\Omega_{N,\delta}$
\begin{equation}\label{Th2a}
\mathfrak{t}_N=\sum_{j}\Theta_j^{(N,\delta)}\Gamma_j^{(N,\delta)} + R_N,\
\ \text{where}\ \  R_N\ge 0 \ \ \text{and}\ \ \bE(1_{\Omega_{N,\delta}}R_N)=\cO(\delta^{1-s}),
\end{equation}
$$\mathfrak{u}_N=\sum_{j}\Theta_j^{(N,\delta)} + \hR_N, \ \hbox{ where }
\hR_N\ge0, \ \bE(\hR_N)=\cO(\delta^{1-s}).$$

\noindent (b) If $s=1$ then for $\omega\in\Omega_{N,\delta}$ and a given $1/2<\kappa<1$
$$\mathfrak{t}_N=\sum_{j}\Theta_j^{(N,\delta)}(\Gamma_j^{(N,\delta)}-1) + R_N,\ \ \text{where}\ \  \bE\left[1_{\Omega_{N,\delta}}\mathbb{E}_{\omega}(R_N^2)\right]^\kappa=\cO(\delta^{2\kappa-1}),$$
$$\mathfrak{u}_N=\sum_{j}\Theta_j^{(N,\delta)} - \brc|\ln\delta|+ \hR_N,\ \ \text{where}\ \
\bE(|\hR_N|^{2})=\cO(\delta).$$

\noindent(c) If $1< s<2$ then for $\omega\in\Omega_{N,\delta}$
$$\mathfrak{t}_N=\sum_{j}\Theta_j^{(N,\delta)}(\Gamma_j^{(N,\delta)}-1) + R_N,\ \ \text{where}\ \  \bE\left[1_{\Omega_{N,\delta}}\mathbb{E}_{\omega}(R_N^2)\right]=\cO(\delta^{2-s}),$$
$$\mathfrak{u}_N=\sum_{j}\Theta_j^{(N,\delta)}-\frac{\brc}{(s-1)\delta^{s-1}} + \hR_N,\ \ \text{where}\ \
\bE(\hR_N^2)=\cO(\delta^{2-s}).$$
\end{theorem}
\begin{remarks} 1. Theorem \ref{ThMain} was proven in \cite{DG} for SRWRE.
%\cite{DG} except that \cite{DG} deals only with the SRWRE on the line.
The next two remarks are similar to those  following Theorem 2 in \cite{DG};
we nevertheless believe that they are worth of being repeated.

2. The estimates of the remainders in the statements of Theorem \ref{ThMain}
hold for all $\delta>0 $ but are not uniform in $N$. More precisely,
e. g. the relation $\bE(|\hR_N|^{2})=\cO(\delta)$
in (b) means that for any $\delta>0$ there is $N_\delta$ and a constant $C$ (which does not depend on $\delta$) such that
$\bE(|\hR_N|^{2})\le C\delta$ if $N>N_\delta$.

3. The dependence of $\Theta^{(N,\delta)}$ on $\omega$ persists as $N\to\infty$
whereas  $\Gamma^{(N,\delta)}$ becomes ``almost" independent of $\omega$. More precisely, for
$K\gg 1$ and sufficiently large $N$ the events $B_k^N:=\{|\Theta^{(N,\delta)}|=k\}$, $0\le k\le K$,
form, up to a set of a small probability, a partition of $\Omega$.
If $\omega\in B_k^N$ then $\Gamma^{(N,\delta)}\equiv \Gamma^{(N,\delta)}(\omega, X)$ is a collection
of $k$ random variables which converge weakly as $N\to\infty$ to a collection of $k$ i.i.d. standard
exponential random variables.
Thus the only dependence of $\Gamma^{(N,\delta)}(\omega, X)$ on $\omega$ and $\delta$ which persists
as $N\to\infty$ is reflected by the fact that
$|\Theta^{(N,\delta)}|=|\Gamma^{(N,\delta)}|$. (Remember that $X$ is the trajectory
of the walk and the purpose of our notation is to emphasize the dependence of
$\Gamma^{(N,\delta)}$ on both $\omega$ and $X$.)
\end{remarks}

The following statement is a corollary of Theorem \ref{ThMain}. In the case of the SRWRE, Theorem
\ref{ThAnn} is one of the main results of \cite{KKS}.
\begin{theorem}
\label{ThAnn}
The annealed walk has the following properties:
\newline
(a) If $s<1$ then the distribution of $\frac{{T}_N}{N^{1/s}}$ converges to a stable law with index $s.$
\newline
(b) If $1<s<2$ then there is a constant $u$ such that the distribution of
$\frac{{T}_N-N u}{N^{1/s}}$ converges to a stable law with index $s.$
\newline
(c) If $s=1$ then there is a sequence $u_N\sim c N \ln N$ (defined as in \eqref{ExpectationT})
such that the distribution of
$\frac{{T}_N-u_N}{N}$ converges to a stable law with index $1.$
\end{theorem}
The proof of this theorem will not be given because its derivation from Theorems \ref{ThTraps}
and \ref{ThMain} is easy (cf Lemma \ref{LmPSt} in Appendix \ref{SSPP}) and also was carried out in \cite{DG}.

\section{The asymptotic properties of traps and roof of Theorem \ref{ThTraps}.}\label{proofTh5}
\subsection{Auxiliary Lemmas. }The proof of Theorem \ref{ThTraps} requires understanding
of the asymptotic behaviour of traps. The following 5 lemmas describe the properties of
traps that shall be used in the sequel.

We start with $w_n$ defined by \eqref{DefW}. Observe that we have
\begin{equation}
\label{WRec}
w_n=\lambda_n w_{n+1}+(\pi_n, v_{n}).
\end{equation}

\begin{lemma}
\label{Lm2Clusters}
There exist $\eps_1>0, \eps_2>0, 0<\beta<1$ and $C>0$ such that for any $(\brzeta, \brpi, \brv)$
 we have

 (a) If $k\leq\eps_1 \ln t$ then
$$ \bP(w_n\geq t, w_{n+k}\geq t|\zeta_n=\brzeta, \pi_n=\brpi, v_n=\brv)
\leq
\frac{C\beta^k}{t^s}; $$
(b) If $k\geq\eps_1 \ln N$ then
$$ \bP(w_n\geq t, w_{n+k}\geq t|\zeta_n=\brzeta, \pi_n=\brpi, v_n=\brv)\leq
C t^{-s(\eps_2+1)}. $$
\end{lemma}

\begin{proof}
For brevity we shall denote $\brbP=\bP(\cdot|\zeta_n=\brzeta, \pi_n=\brpi, v_n=\brv).$

(a) From \eqref{WRec} we have
$$ w_{n}=\lambda_{n}\dots \lambda_{n+k-1}+O(K^{k})  .$$
So if $K^k\ll t $ then
\begin{align*}
\brbP(w_n\geq t, w_{n+k}\geq t) & \leq &
\brbE(1_{\lambda_n \dots \lambda_{n+k-1}\geq 1/2} P_{\zeta_{n+k}, \pi_{n+k}, v_{n+k}}(w_{n+k} \geq t)) \\
& \leq &
\frac{\brbE(1_{\lambda_n \dots \lambda_{n+k-1}\geq 1/2} f(\zeta_{n+k}, \pi_{n+k}, v_{n+k}))}{t^s} \\
& \leq &
\frac{C \brbP(\lambda_n \dots \lambda_{n+k-1}\geq 1/2) }{t^s} \leq \frac{C\beta^k}{t^s}.
\end{align*}
(b) Consider two cases

(I) $k>\brC\ln t$ where $\brC$ is the constant from Lemma \ref{LmWRen}. Then
$$ \brbP(w_n\geq t, w_{n+k}\geq t)
\leq \brbP(w_{n, n-k-1}\geq t-1, w_{n+k}\geq t)+\brbP(w_n-w_{n, n-k-1}>1). $$
The second term is $O(t^{-\brs})$ while the first term equals to
$$ \brbE(1_{w_{n, n+k-1}\geq t-1} \bP_{\zeta_{n-k}, \pi_{n-k}, v_{n-k}} w_{n+k}\geq t)
\leq \frac{C}{t^s} \brbP(w_{n, n+k-1}\geq t-1)\leq \frac{C}{t^{2s}}. $$
(II) $\eps_1 \ln t<k\leq \brC \ln t.$ Fix $\teps\ll 1.$ Then
$$ \brbP(w_n\geq t, w_{n+k}\geq t)\leq \brbP(w_{n+k}\geq t^{1+\teps})+
\brbP(w_n\geq t, t\leq w_{n+k}\leq t^{1+\teps}). $$
The first term is $O(t^{-(1+\teps)s})$ while the second term is less than
$$ \brbP(\lambda_n\dots \lambda_{n+k-1}\geq t^{-\teps/2}, w_{n+k}\geq t)+
\brC\ln t \max_{1<\leq j<k} \brbP(\lambda_n \dots \lambda_{n+j}\geq t^{1-\teps}, w_{n+k}\geq t). $$
Both terms are estimated in the same way so we only discuss the first one
$$  \brbP(\lambda_n\dots \lambda_{n+k-1}\geq t^{-\teps/2}, w_{n+k}\geq t)\leq
\brbE(1_{\lambda_n \dots \lambda_{n+k-1}\geq t^{\teps/2}} \bP_{\zeta_{n-k}, \pi_{n-k}, v_{n-k}} (w_{n+k}\geq t) $$
$$\leq \frac{C}{t^s} \brbP(\lambda_n \dots \lambda_{n+k-1}\geq t^{\teps/2})
\leq C t^{s(1+\eps_2)} $$
as claimed.
\end{proof}

It may happen that not all massive sites belong to one of the clusters.
This situation is controlled by the following
\begin{lemma}\label{non-cluster}
There is $\beta<1$ such that for $n\in[0,N-1]$
\begin{equation}
\label{Eq4Peaks}
\bP\left(\rho_{n}\geq \delta
N^{1/s} \hbox{ and $n$ is not in a trap}  \right)\le \Const \frac{\beta^M}{N} .
\end{equation}
\end{lemma}

\begin{proof} Suppose that  $n$ is a massive point which is not in a trap.
Then consider all massive points $n_i$ such that $n<n_1<...<n_k<n+M$. Note that
such points exist because otherwise $n$ would have been a marked point.
Let now $n^*>n_k$ be the nearest to $n_k$ massive point. Then by construction  $n^*\ge n+M$.
Also $n^*\le n+2M$ because otherwise $n_k$ would have been a marked point and
$n$ would belong to the $n_k$-trap. Hence the event
$$\{n \hbox{ is massive and not in a trap}\}
\subset\bigcup_{n'\in[n+M, n+2M]}\{w_{n}\geq \delta
N^{1/s},\,w_{n'}\geq \delta
N^{1/s}\}.$$
By Lemma \ref{Lm2Clusters}(b) we obtain
$$\begin{aligned}
&\bP\left(n \hbox{ is massive and not in a trap} \right)\\
&\le \sum_{n'=n+M}^{n+2M}\bP\left(
w_{n}\geq \delta
N^{1/s},\, w_{n'}\geq \delta
N^{1/s}\right)\le\Const \frac{\beta^M}{N}
\end{aligned}
$$
which proves our statement.
\end{proof}

Our next goal is to show that $\bP(n \text{ is massive})$ and
$\bP(n \text{ is marked})$ are of the same order.

\begin{lemma}
\label{LmLEss}
(a) For each $l,R\geq 1$ the following limit exists
$$f_l(\brzeta, \brpi, \brv, R)
=\lim_{t\to\infty} t^s \bP(w_n\geq t R, w_{n+j}<t \text{ for }j=1\dots l|
\zeta_n=\brzeta, \pi_n=\brpi, v_n=\brv).$$
(b) Let\footnote{the limit exists since $f_l$ is decreasing} $\brf=\lim_{l\to\infty} f_l(\brzeta, \brpi, \brv, R).$
Then
$|\brf(\brzeta, \brpi, \brv,R)-f_l(\brzeta, \brpi, \brv, R)|=O(\theta^l).$
\end{lemma}

\begin{proof}
By \eqref{WRec}
$$ \brbP(w_n\geq t R, w_{n+j}<t \text{ for }j=1\dots l)\approx $$
$$\brbE\left(\bP_{\zeta_{n+l}, \pi_{n+l}, v_{n+l}}\left(
w_{n+l} \in \left[\frac{tR}{\lambda_n \dots \lambda_{n+l-1}},
\frac{t}{\max_j(\lambda_n \dots \lambda_{n+j-1})}\right]\right)\right) $$
so the result follows from Lemma \ref{LmWRen}.

(b) Since
$$  \brbP(w_n\geq Rt, w_{n+j}<t \text{ for }j=1\dots l)-
 \brbP(w_n\geq Rt, w_{n+j}<t \text{ for }j=1\dots {l+1})$$
$$\leq \brbP(w_n\geq t, w_{n+l+1}\geq t) $$
the result follows by  Lemma \ref{Lm2Clusters}(a).
\end{proof}

\begin{lemma}
\label{LmFPos}
$\bP(\brf(\zeta, \pi, v,1)>0)>0.$
\end{lemma}
\begin{proof}
Assume to the contrary that $\bP(\brf>0)=0.$
Then for each $\eps$ there is $n_0$ such that for $N\geq N_0$ we have
$\bP(n \text{ is marked})\leq \frac{\eps}{N}.$
Combining Lemma \ref{LmWRen} and Lemma \ref{non-cluster} we obtain that there is a constant $c$ such that
\begin{equation}
\label{TnOrderN}
\bP(\bbT_n)\geq \frac{c}{N}, \text{ where } \bbT_n=\{n \text{ is marked and belongs to a trap}\}.
\end{equation}
If $n$ is in a trap let $D_n$ be the distance to the nearest marked point to the left of $n.$
Given $D$ we write
$$ \bP(\bbT_n)=\bP(\bbT_n \text{ and }D_n<D)+
\bP(\bbT_n \text{ and }D_n\geq D). $$
The first term equals to
$$ \sum_{j=0}^{D-1} \bP(n \text{ is massive and } n+j \text{ is marked})
\leq \sum_{j=0}^{D-1} \bP(n+j \text{ is marked})\leq \frac{\eps D}{N}. $$
On the other hand
$$ \bP(\bbT_n \text{ and }D_n\geq D)\leq \sum_{j=D}^{M} \bP(n \text{ and } n+j\text{ are massive})
\leq \frac{C\beta^D}{N}. $$
Thus
$$ P(\bbT_n)\leq \frac{\eps D+C\beta^D}{N}. $$
Taking $D$ so large that $C\beta^D\leq \frac{c}{3} $ and then choosing $\eps\leq \frac{c}{3D}$
we obtain a contradiction with \eqref{TnOrderN}.
\end{proof}

We now turn our attention to the mass of the trap. Observe that for $j\leq M$ we have
$$ \rho_{n-j}\approx w_n \lambda_{n-j}\dots \lambda_n l_{n-j} (v_{n-j}). $$
Accordingly introduce
$$ a_n=\sum_{j=0}^\infty \lambda_{n-j}\dots \lambda_n l_{n-j} (v_{n-j})
\text{ and }
a_{n_1, n_2}=\sum_{j=0}^{n_2-n_1} \lambda_{n_2-j}\dots \lambda_{n_2} l_{n_2-j} (v_{n_2-j}) .
$$
In the next result proven in Section \ref{tail} we use the same notation as in Lemma \ref{LmWRen}.
\begin{lemma}
\label{LmARen}
(a) If $k>\brC \ln t $ then
$$ \bP(a_{n}-a_{n-k, n}\geq 1)\leq C t^{-\brs} \text{ and}$$

$$(b)\quad \bP(a_{n-k, n}\geq t)\sim \hc t^{-s} .\Quad $$
(c) Moreover there exists a measure $\nu$ such that if $k>\brC \ln t$ then
$$ t^s \bE\left(1_{a_{n-k, n}\geq t} f\left(\zeta_n, \pi_n, v_n, \frac{a_{n-k, n}}{t}\right)\right)\sim
\hc \iint f(\zeta, \pi, v, z) \frac{dz d\nu(\zeta, \pi, v)}{z^{s+1}}. $$
\end{lemma}

\begin{corollary}
\label{CrMassIntensity}
(a) The following limit exists
$$h(t,\delta)=\lim_{N\to\infty} N^{1/s} \bP(n \text{ is marked and } \mm_n>t N^{1/s}).$$
(b) There is $\bc>0$ such that
$ \lim_{\delta\to 0} h(t, \delta)=\bc t^{-s}. $
\end{corollary}

\begin{proof}
Take $W$ such that if $n$ is marked then $w_n\in [\delta N^{1/s}, \delta W N^{1/s}]. $
We have
$$ \bP (n \text{ is marked and } \mm_n>t N^{1/s})=
\bE\left(1_{a_{n-M, n}>\frac{t}{\delta W}}
\bP_{\zeta_n, \pi_n, v_n} \left(w_n>\min\left(1, \frac{t}{\delta a_n}\right) \delta N^{1/s} \right)\right)$$
$$ \sim \frac{1}{\delta^s N} \bE\left(1_{a_{n-M, n}>\frac{t}{\delta W}}
\brf\left(\zeta_n, \pi_n, v_n, \min\left(1, \frac{\delta a_n}{t}\right)\right)\right) . $$
This proves (a). To prove (b) we use Lemma \ref{LmARen}(c) to get
$$ \frac{1}{\delta^s} \bE\left(1_{a_{n-M, n}>\frac{t}{\delta W}}
\brf\left(\zeta_n, \pi_n, v_n, \min\left(1, \frac{\delta a_n}{t}\right)\right)\right)$$
$$\sim \hc \frac{W^s}{t^s} \iint \brf(z,\pi, v, \min(1, W/z)) \frac{d\nu dz}{z^s} .
 $$
\end{proof}

\subsection{Proof of Theorem \ref{ThTraps}.}
We are now in a position to prove Theorem \ref{ThTraps}.
The following lemma essentially repeats the statement of Theorem \ref{ThTraps}
with the difference that we can now state it in terms of $h(\cdot, \delta)$ studied above
(Corollary \ref{CrMassIntensity}).
\begin{lemma} Suppose that all conditions of Theorem \ref{ThTraps} are satisfied. Then

(a) For a fixed $\delta>0$ the point process
$$\{(nN^{-1},\mm_nN^{-\frac{1}{s}}):\, n \text{ is }\delta N^{\frac{1}{s}}\text{-marked}\}$$
converges as $N\to\infty$ to a Poisson process on $[0,1]\times [\delta, \infty)$ with measure
$dt'\mu_\delta$ such that $\mu_\delta([t_1, t_2])=h(t_2, \delta)-h(t_1, \delta).$

(b) As $\delta\to 0$ $\mu_\delta$ converges to a measure with density $\frac{\bc s}{t^{s+1}}.$
\end{lemma}

\begin{proof}
To prove (a) we use Bernstein's big block-small block method. Namely, we divide
$[0, N]$ into big blocks $I_j$ of length $N^{\eps_3}$ separated by small blocks of length
$N^{\eps_3/2}.$  We take $\eps_3<\eps_2$, where $\eps_2$ is the constant from Lemma
\ref{Lm2Clusters}. By Lemma \ref{LmWRen} the probability that there is a trap inside the
union of the short intervals tends to $0$ so it suffices to consider the union of long intervals.
We claim that for each $j$ the probability that $I_j$ contains two or more traps is $o(N^{\eps_3-1}).$
Indeed due to Lemma~\ref{Lm2Clusters}, the above probability is bounded by
$$ \sum_{n\in I_j} \sum_{M_N <k<N^{\eps_3}} \bP(\text{both }n \text{ and }n+k \text{ are marked}) $$
$$ \leq \sum_{n\in I_j} \sum_{M_N <k<N^{\eps_3}} \bP(\text{both }n \text{ and }n+k \text{ are massive})$$
$$ \leq \frac{C}{N} \sum_{n\in I_j}
\left[\sum_{k=M_N}^{\eps_1 \ln N} \beta^k+\sum_{k=\eps_1 \ln N}^{N^{\eps_3}} N^{-\eps_2} \right]
 \leq C N^{\eps_3-1} \left[\beta^{M_N}+N^{\eps_3-\eps_2}\right] $$
proving the claim.

Next, let $\{J_l\}$ be a collection of $d$ non-intersecting open intervals in $[\delta, \infty)$ and
$K=(t_1',t_2')$ be an open interval from $[0,1]$.  Let
$$N_l=\Card(n\leq N: n \text{ is marked and }( n,\mathfrak{m}_n)\in NK\times N^{1/s} J_l ). $$
We need to show that, as $N\to\infty$, $\{N_l\}$ converge to independent random
variables having Poisson distribution with parameter $(t_2'-t_1')\mu_\delta(J_l)$
(the proof for all other finite collections of open quadrangles in
$[0,1]\times[\delta,\infty)$ easily follows from this case).
We%Lemmas \ref{LmWRen} and \ref{LmARen} imply that
shall now replace $\mm_n$ by $\widetilde{\mm}_n$ which are defined similarly to $\mm_n$
but have the property of being i.i.d. random variables and, at the same time,
$|\mm_n-\widetilde{\mm}_n|\le N^{-\frac{100}{s}}$. Namely, we do the following:

1. Define Markov process $(\psi_n, \hat{\pi}_{n})$  starting with
initial conditions $\hat{\pi}_{n_j}$ and $\psi_{n_j}$ which have all entries equal to $\frac{1}{m}$,
where $n_j$ is the middle of the short interval preceding $I_j.$ This process is defined
for $n>n_j$ with the $\psi$ component given by \eqref{EqPsi} and $\hat{\pi}_{n}=\hat{\pi}_{n-1}\psi_{n-1}$.

2. Similarly to \eqref{ExpectF10}, set for $n\in I_j$
$$
\hat{\rho}_n:= \sum_{j=n}^{n_{j+1}} \hat{\pi}_j \hat{A}_j \dots \hat{A}_{n+1}\hat{\bu}_n,
$$
where $\hat{A}_{j},\ \hat{\bu}_n$ are define as the corresponding $A$s and $\hat{\bu}$s with $\zeta$s replaced by
$\psi$s.

3. Finally $\widetilde{\mm}_n$ is defined similarly to $\mm_n$ but with $\hat{\rho}_n$ replacing $\rho_n$.

The independence mentioned above is obvious from the construction and the approximation property
now follows from Proposition \ref{propD1}:
$$||\zeta_n-\psi_n||\le \Const \theta^{n-n_j}.$$

Let now $\{\Gamma_j\}$ be a sequence of random $d$-dimensional vectors such that
$$\Gamma_j=
\begin{cases} e_l & if
I_j\subset NK, \text{ there is exactly one trap } \tn\in I_j,\ \mm_\tn\in N^{\frac{1}{s}}J_l,\\
0 &\text{ in all other cases.}
\end{cases}
$$
Then $\Gamma_j$ are iid random vectors and
$$\bP(\Gamma_l=e_l)\sim (t_2'-t_1')\mu_\delta(J_l) N^{\eps_3-1}, \quad
\bP(\Gamma_l=0)= 1-(t_2'-t_1')\mu_\delta(J_l) N^{\eps_3-1}+o(N^{\eps_3-1}). $$
Therefore part (a) follows from the Poisson Limit Theorem for independent random vectors. Part (b)
follows from part (a) and Corollary \ref{CrMassIntensity}(b).
\end{proof}

\section{Proof of Theorem \ref{ThMain}.}\label{proofTh6}
The proof of Theorem \ref{ThMain} is to an extent similar to that of Theorem 2 in
\cite{DG} and we shall give only an outline of it putting emphases at those parts which are new.
%There are at least two reasons for that: first of all this simplifies the reading of the paper,
%and secondly it is both a relatively short and a very important part of the paper.
As in \cite{DG}, we start with a lemma which allows us to show the smallness of the contribution
to $T_N$ which comes from the sites where $\rho_n<\delta N^{\frac{1}{s}}$.
We then compute the main contribution to $T_N$ which comes from the traps in $[0,N]$
described in Theorem \ref{ThTraps}.

Within this section we shall use  the following notation: $\xi_n=\sum_{i=1}^m \xi_{(n,i)}$. Obviously,
$\rho_n=\mathbb{E}_\omega(\xi_n)$. In these notations, Lemma \ref{LmLow} becomes an exact copy of
Lemma 4.1 from \cite{DG}.

\begin{lemma}
\label{LmLow}
Let $\delta> 0$. Then there is $N_\delta$ such that
for $N>N_\delta$ the following holds:
%Let $\delta\equiv\delta_N\to 0$ as $N\to\infty.$ Then for any $\eps>0$ the following holds:

(a) If $0<s<1$ then
$$ \mathrm{E}\left(\sum_{w_n<\delta N^{1/s}}\xi_{n}\right)\le \Const N^{1/s}\delta^{1-s}. $$

(b) If $1<s<2$ then
there is a set $\tOmega_{N,\delta}$ such that $\bP(\tilde{\Omega}_{N,\delta}^c)\leq N^{-100}$ and
$$ \bE\left(1_{\tOmega_{N, \delta}}
\mathbb{E}_\omega\left(\sum_{w_n<\delta N^{1/s}} (\xi_n-\rho_n)\right)^2\right)
\le \Const N^{2/s}\delta^{2-s}. $$

(c) If $0<s<1$ then
$$ \bE\left(\sum_{w_n<\delta N^{1/s}} \rho_n\right)\le \Const N^{1/s}\delta^{1-s}. $$

(d) If $1<s<2$ then
$$ \Var\left(\sum_{w_n<\delta N^{1/s}} \rho_n\right)\le
\Const N^{2/s}\delta^{2-s}. $$

(e) If $s=1$ then given $\frac{1}{2}<\kappa<1$ there is a set $\tOmega_{N, \delta}$ such that
$\bP(\tilde{\Omega}_{N,\delta}^c)\leq N^{-100}$ and
\begin{equation}
\label{LLe1}
\bE\left(1_{\tOmega_{N,\delta}} \left(\Var_\omega \left(
\sum_{w_n<\delta N} (\xi_n-\rho_n)\right)\right)^\kappa\right)
\le \Const N^{2\kappa} \delta^{2\kappa-1},
\end{equation}
\begin{equation}
\label{LLe2}
\bE\left(\left(\sum_{w_n<\delta N} \left(\rho_n-\bE\left(\rho I_{\rho<\delta N}\right)\right)
\right)^{2}\right)
\le \Const N^{2} \delta.
\end{equation}
\end{lemma}
\begin{proof} Parts (a) and (c) follow from Lemma \ref{LmWRen} and Markov inequality (cf. the proof of Lemma 4.1 in \cite{DG}).
%will not be discussed here
%because they are simple and don't differ from the proofs of parts (a) and (c) of Lemma 4.1 in \cite{DG}.

The proofs of (b), (d), and (e) in Lemma 4.1 in \cite{DG} do not go through directly in the case of
the strip. We shall give a complete proof of (b) and the required adaptations in the cases (d) and (e)
are the same.

\textit{Proof of} (b). Set $\chi_n=
I_{w_{n}<\delta N^{1/s}}$; this notation will be used only
within the proof of Lemma \ref{LmLow}. Denote $\tY_\delta=\sum_{w_n<\delta N^{1/s}} (\xi_n-\rho_n).$ Then
$\mathbb{E}_\omega(\tY_n)=0$ and so it suffices to show that
$\Var_\omega(\tY_\delta)=O(\delta^{2-s} N^{2/s})$
except for $\omega$ from a set of small probability.
It follows from Lemma \ref{return} that if $K$ is sufficiently large, $n_2-n_1>K\ln N$,  and
$\omega\not\in \Omega_{n_1, K\ln N, \theta_0}$  (see \eqref{returnOmega} for the definition of this set) then
$$\Cov_\omega\left(\xi_{n_1}, \xi_{n_2}\right)\le CN^{-100}.$$
Therefore
\begin{equation}
\begin{aligned}
\label{VAR}
\Var_\omega(\tY_\delta)
&=\left|o(1)+\sum_{n_2-K\ln N<n_1<n_2} 2\chi_{n_1} \chi_{n_2}
\Cov_\omega\left(\xi_{n_1}, \xi_{n_2}\right)
+ \sum_{n} \chi_n \Var_\omega\left(\xi_{n} \right)\right|\\
&\leq
1+\Const\sum_{n_2-K\ln N<n_1\leq n_2} \rho_{n_1} \rho_{n_2} \chi_{n_1} \chi_{n_2} \\
\end{aligned}
\end{equation}
where the summation is over pairs with $\rho_{n_i}<\delta N^{1/s}$. The last step uses
$$\left|\Cov_\omega\left(\xi_{n_1}, \xi_{n_2}\right)\right| \leq
\sqrt{\Var_\omega\left(\xi_{n_1}\right) \Var_\omega\left(\xi_{n_2}\right)}
\leq C\rho_{n_1} \rho_{n_2}.$$
Here, apart of the Cauchy-Schwartz inequality, we use the fact that if the walk
starts from $z\in L_{n_0}$, then $\mathbb{P}_{\omega,z}\{\xi_{n,j}\ge1\}\ge\varepsilon_0$
for $n\ge n_0$ (the existence of $\varepsilon_0$ is due to \eqref{EqC2*}). The latter
inequality implies that $\sqrt{\Var_\omega\left(\xi_{n,i})\right)}
\le \Const \rho_{n,i}$ with the constant depending only on $\varepsilon_0$.
We use here elementary explicit expressions for all involved quantities, see \eqref{ExpVar}.
%the fact that $\xi_{n,i}$ conditioned on the walk visiting $(n,i)$ has
%geometric distribution and one has .)

Next, we have to estimate the expectation of the last sum in \eqref{VAR}.
To this end introduce
$$\mathfrak{r}_n=1+||A_{n+1}||+||A_{n+2}A_{n+1}||+...\equiv\sum_{j=0}^{\infty}||A_{n+j}...A_{n+1}||.$$
It is clear from \eqref{ExpectF10} and the strong ellipticity condition \eqref{EqC2*} that there are constants
$c_1,\ c_2$ such that $ c_1\mathfrak{r}_n<w_n< c_2\mathfrak{r}_n $. Hence there is a $C$ and $t_0$ such that
for $t>t_0$ uniformly in $\zeta$
\begin{equation}\label{estimate}
\mathbf{P}\{\mathfrak{r}_n > t|\zeta_n=\zeta\}\le Ct^{-s}\text{ and }
\mathbf{E}\left[\left(\mathfrak{r}_n^2|\zeta_n=\zeta\right)
I_{(\mathfrak{r}_n|\zeta_n=\zeta)\le t}\right] \le \Const\, t^{2-s}.
\end{equation}
We also have that for $k>0$
$$\rho_{n-k}\le C\sum_{j=0}^{\infty}||A_{n-k+j}...A_{n-k+1}||\le
C\sum_{j=0}^{k-1}||A_{n-k+j}...A_{n-k+1}||+C||A_{n}...A_{n-k+1}||\mathfrak{r}_n.$$
and
$$
\rho_{n-k}\rho_n \le C\sum_{j=0}^{k-1}||A_{n-k+j}...A_{n-k+1}||\mathfrak{r}_n+C||A_{n}...A_{n-k+1}||\mathfrak{r}_n^2.
$$
To estimate $\mathbf{E}(\mathfrak{r}_{n-k}\mathfrak{r}_n)$ we condition on $\zeta_n$ and use the fact that
the conditioned random variables $||A_{n-k+j}...A_{n-k+1}|||\zeta_n$ and $\mathfrak{r}_n|\zeta_n$ are independent.
Therefore
$$
\bE\left(||A_{n-k+j}...A_{n-k+1}||\mathfrak{r}_n\right)=
\bE\left[\bE\left(||A_{n-k+j}...A_{n-k+1}|||\zeta_n\right)\bE\left(\mathfrak{r}_n|\zeta_n\right)\right]\le
C\beta^j,
$$
where $\beta:=r(1)<1$  and $\bE\left(\mathfrak{r}_n|\zeta_n\right)<\Const$ since $s>1$.

Similarly, but this time using also \eqref{estimate} we obtain
$$
\begin{aligned}
\bE\left(||A_{n}...A_{n-k+1}||\mathfrak{r}_n^2\chi_n\right)&\le
\bE\left[\bE\left(||A_{n}...A_{n-k+1}|||\zeta_n\right)\bE\left(\mathfrak{r}_n^2
I_{(\mathfrak{r}_n|\zeta_n=\zeta)\le C \delta N^{\frac{1}{s}}}|\zeta_n\right)\right]\\
&\le C\beta^k\delta^{2-s}N^{\frac{2-s}{s}}.
\end{aligned}
$$
Hence
$$
\bE\left(\rho_{n-k}\rho_{n} \chi_{n}\right) \le
\Const\left(1+\beta^k\delta^{2-s}N^{\frac{2-s}{s}}\right)
$$
and therefore for $N>N_\delta$
$$
 \bE\left(\sum_{k=0}^{K\ln N}(\rho_{n-k}\rho_{n} \chi_{n})\right)
\leq \Const \left(\ln N +\delta^{2-s}N^{\frac{2-s}{s}}\right)\le C\delta^{2-s}N^{\frac{2-s}{s}}.
$$
Finally
$$
\begin{aligned}
\bE\left(\sum_{n_2-K\ln N<n_1\le n_2} \rho_{n_1} \rho_{n_2}\chi_{n_1}\chi_{n_2}\right)
&\leq \bE\left(\sum_{n_2-K\ln N<n_1\le n_2} \rho_{n_1} \rho_{n_2}\chi_{n_2}\right)\\
&\leq C N\delta^{2-s}N^{\frac{2-s}{s}}=C \delta^{2-s} N^{2/s}.
\end{aligned}
$$
Part (b) is proven. The proofs of parts (d) and (e) follow the proof of the corresponding statement of Lemma 4.1 in \cite{DG}
with the modifications similar to ones presented here.
\end{proof}
We are now prepared to explain the main further steps of the proof.

We consider the case $s\in (0,1)$; other cases are treated similarly.
Present the time spent by the walk in $[L_0,L_{N-1}]$ as
\begin{equation}\label{TN}
T_N=\sum_{n=0}^{N-1}\xi_{n}\equiv\sum_{n=0}^{N-1}\sum_{i=1}^{m} \xi_{(n,i)}=S_1+\ S_2+\ S_3,
\end{equation}
where
$$\begin{aligned}
S_1=&\sum_{n:\,w_n<\delta N^{1/s},\,n\,\not\in\,\text{any trap}} \xi_n \\
S_2=&\sum_{n:\, w_n\ge\delta N^{1/s},\,n\text{ is not in a trap}}\xi_n\\
S_3=&\sum_{n:\,n\text{ is in a trap}}\xi_n.
\end{aligned}
$$
By Lemma \ref{LmLow}, (a) we have that $\bE(S_1)\le\Const N^{1/s}\delta^{1-s}$.

Next, denote
$$\brOmega_{N,\delta}^{(1)}:=\{\omega: \text{ there is }n\in[0,N-1] \text{ s. t. }
w_n\ge\delta N^{1/s},\,n \text{ is not in a trap}\}.$$
It follows from \eqref{Eq4Peaks} that
$$\bP\left(\brOmega_{N,\delta}^{(1)}\right)\le N\bP\left(w_{n}\geq \delta
N^{1/s} \hbox{ and $n$ is not in a trap}  \right)\le\Const \beta^M.$$
But then
$$
\bP(S_2\not=0)\le \bP\left(\brOmega_{N,\delta}^{(1)}\right)\le\Const \beta^M\to0\text{ as } N\to\infty.
$$
We thus have that for $\omega\not\in\brOmega_{N,\delta}^{(1)}$
$$
\mathfrak{t}_N=N^{-\frac{1}{s}}S_3+N^{-\frac{1}{s}}S_1=N^{-\frac{1}{s}}S_3+R_N,
$$
where $R_N:=N^{-\frac{1}{s}}S_1$ and satisfies the requirements of (a), Theorem \ref{ThMain}.

It remains to analyze $S_3$ which is the main contributor to $T_N$ coming from the sum
over the traps in $[L_0,L_{N-1}]$. Let us present it as follows:
$$
N^{-\frac{1}{s}}S_3=\sum_{n:\, n\text{ is marked}}N^{-\frac{1}{s}}\sum_{j=0}^M \xi_{n-j}.
$$
Since $n$ is marked, we can choose a $k$ such that $\rho_{n,k}\ge m^{-1}\delta N^{\frac{1}{s}}$.
Now present
$$
\sum_{j=0}^M \xi_{n-j}=\sum_{j=0}^M\sum_{i=1}^m
\left(\frac{\xi_{n-j,i}}{\rho_{n-j,i}}-\frac{\xi_{n,k}}{\rho_{n,k}}\right)\rho_{n-j,i}+
\frac{\xi_{n,k}}{\rho_{n,k}}\sum_{j=0}^M \rho_{n-j}.
$$
Next, we shall use Corollary \ref{CorA2} to estimate $\left\|\frac{\xi_{n-j,i}}{\rho_{n-j,i}}-
\frac{\xi_{n,k}}{\rho_{n,k}}\right\|$, where $\|f\|:=\sqrt{\mathbb{E}_{\omega}(|f|^2)}$ for a function $f$
on the space of trajectories of the walk. We have:
$$
\left\|\frac{\xi_{n-j,i}}{\rho_{n-j,i}}-\frac{\xi_{n,k}}{\rho_{n,k}}\right\|\le
\sum_{r=n-j}^{n-1} \left\|\frac{\xi_{r,i}}{\rho_{r,i}}-\frac{\xi_{r+1,i}}{\rho_{r+1,i}}\right\|
+\left\|\frac{\xi_{n,i}}{\rho_{n,i}}-\frac{\xi_{n,k}}{\rho_{n,k}}\right\|
\le C \sum_{r=n-j}^{n-1} \frac{1}{\sqrt{\rho_{r,i}}}+\frac{C}{\sqrt{\rho_{n,k}}}.
$$
Condition \eqref{EqC2*} together with \eqref{ExpectF} imply that there is $\eps_0>0$ such that
$\rho_{n-j,i}\ge \eps_0 ||A_{n-j} \dots A_{n}|| \rho_{n,k}$. Hence for $n-j$ belonging to a trap, that is $(n-j)\in[n-M,n]$,  we have that
$\rho_{n-j,i}\ge c\eps_0^M \rho_{n,k}\ge cN^{-\bar{\eps}}\rho_n$. (Remember that $M=\ln\ln N$
and therefore these inequalities hold for any $\bar{\eps}>0$ and $N>N_{\bar{\eps}}$.) Thus
$$
\left\|\sum_{j=0}^M\sum_{i=1}^m
\left(\frac{\xi_{n-j,i}}{\rho_{n-j,i}}-\frac{\xi_{n,k}}{\rho_{n,k}}\right)\rho_{n-j,i}\right\|
\le \Const \frac{N^{\bar{\eps}/2}}{\sqrt{\rho_{n}}}\sum_{j=1}^M M \rho_{n-j}
\le \Const \frac{N^\breps}{\sqrt{\rho_{n}}}\sum_{j=1}^M \rho_{n-j}
$$
If for $n$ marked we set
$$
\mathfrak{a}_n=\mathfrak{m}_n^{-1}
\sum_{j=0}^M\sum_{i=1}^m
\left(\frac{\xi_{n-j,i}}{\rho_{n-j,i}}-\frac{\xi_{n,k}}{\rho_{n,k}}\right)\rho_{n-j,i}
$$
then $\|\mathfrak{a}_n\|\le \Const \frac{N^\breps}{\sqrt{\rho_{n}}}\to0$ as $N\to\infty$
and we have
$$\frac{\sum_{j=0}^M \xi_{n-j}}{N^{1/s}}=\left(\frac{\xi_{n,k}}{\rho_{n,k}}
+\mathfrak{a}_n\right) \frac{\mathfrak{m}_n}{N^{1/s}}.$$
Set $\Gamma^{(N,\delta)}_j=\frac{\xi_{n_j,k_j}}{\rho_{n_j,k_j}}
+\mathfrak{a}_{n_j}$, where $\{n_j\}$ is the collection of marked points.
To finish the proof of statement (a) from Theorem \ref{ThMain},
it remains to check that
$$\{\xi_{n_j,k_j}/\rho_{n_j,k_j}\}_{n_j \text{ is marked}}$$
form a collection of asymptotically independent random variables which also
are asymptotically exponential with mean 1.

The convergence to the exponential distribution is an immediate corollary of two facts:
the conditional random variable $\xi_{n,k}|(\xi_{n,k}\ge1)$ is geometric
and $\mathbb{P}_\omega(\xi_{n,k}=0)\to0$ as $N\to\infty$ (to prove this last assertion apply
Lemma \ref{LemmaA3} with $a=(n,k)$ and $b$ the first point visited by the walker inside layer
$L_n$).

To establish the asymptotic independence remember the construction used in the
proof of Theorem \ref{ThTraps}. We have established there that the marked points
belong to the blocks of length $N^{\eps_3}$ which are separated from each other
by the blocks of length $N^{\frac{\eps_3}{2}}$ and, moreover, there is at most one marked
point in a large block. By Lemma \ref{return}, the $\mathbb{P}_\omega$ probability that the walk
would ever return to block $I_{j-1}$ after having reached $I_j$ is of order
$O\left(\theta_0^{N^{\frac{\eps_3}{2}}}\right)$, where $\theta_0<1$, it follows that
any random variables which are functions of the part of trajectory of the walk
starting at the left end of $I_j$ and restricted to the $N^{\eps_3/2}/2$ neighbourhood
of $I_j$ are independent.

Part (a) of Theorem \ref{ThMain} is proved. Parts (b) and (c) are dealt with in a similar way.

\section{Tail asymptotics.}\label{tail}

\begin{proof}[Proof of Lemma \ref{LmWRen}]
Parts (a) and (b) of the lemma follow from $z^+$ part of Theorem \ref{ThRenContr} applied to
the following Markov process
\begin{equation}
\label{OurMP}
\Phi_n=(\pi_n, v_n, \zeta_n),\quad g_n=(P_{n+1}, Q_{n+1}, R_{n+1}).
\end{equation}
Note that due to \eqref{EqC2*} there exists $\breps$ such that both $A_n$ and $\zeta_n$ map $\bbX$ into
$\bbX_\breps=\{v\in \bbX: v_j\geq \breps\}.$ In \eqref{OurMP} $\pi_n$ and $v_n$ are regarded as elements of $\bbX_\breps$
and $\zeta_n$ is an leement of the set of stochastic matrices. Recall that given $\Phi_n$ and $g_n$ we can reconstruct
$\Phi_{n+1}$ using \eqref{EqZeta}, \eqref{y}, \eqref{DefA} and \eqref{DefV}.

In order to apply Theorem \ref{ThRenContr} we need to check three conditions. The first one is eventual contraction
(equation \eqref{EventualContr}).
Since both $A_n$ and $\zeta_n$ map $\bbX$ into $\bbX_\breps$ we can apply Birkhoff Theorem (see e.g. \cite{L}) which tells us that there is
a constant $\brtheta=\brtheta(\breps)<1$ such that $A_n$ and $\zeta_n$ contract the Hilbert metric on $\bbX_\breps$ at least by factor $\brtheta$
(the contraction of $\pi_n$ part also follows from Lemma \ref{lemma1}). The contraction of $\zeta_n$ part is proven in Appendix \ref{ScIndBC}.

Second, we need to check \eqref{NoNest}. In our setting we have to show that for each $t$ there is $n$ such that
$ \bP(||A_n\dots A_1||>t)>0 .$  If this were false then there would exist $t_0>1$ such that
$||A_n\dots A_1||^\alpha<t_0^\alpha$ for all $n$ with probability 1. This would imply $r(\alpha)\leq 1$
for all $\alpha>0.$ Since $r(0)=1$ and $\ln r(\alpha)$ is
strictly convex we would actually have $r(\alpha)<1$ for all positive $\alpha$ contradicting \eqref{defs}.

Lastly we need to show that \eqref{SpCB} has no solutions. In our setting \eqref{SpCB} takes form
\begin{equation}
\label{SpCBMat}
 e^{i\bru} \fh(\pi_{n-1}, v_{n-1} ,\zeta_{n-1})=e^{i u\ln ||A_n v_{n-1}||} \fh(\pi_{n}, v_{n}, \zeta_{n}).
\end{equation}
Take $(P, Q, R)$ in the support of the environment measure.
Let $\zeta_{(P, Q, R)}$ and $A_{(P, Q, R)}$ be defined by \eqref{EqZeta1} and \eqref{APer} respectively and denote by
$\pi_{(P, Q, R)}$ and $v_{(P, Q, R)}$ the positive eigenvectors of these matrices. Then
$$ \Phi_n\equiv (\pi_{(P, Q, R)} , v_{(P, Q, R)}, \zeta_{(P, Q, R)}), \quad g_n\equiv (P, R, R) $$
is an admissible trajectory. Evaluating \eqref{SpCBMat} along this trajectory we get
$$ e^{i\bru} \fh(\pi_{(P, Q, R)}, v_{(P, Q, R)}, \zeta_{(P, Q, R)})
=e^{i u\ln \lambda_{(P, Q, R)}} \fh(\pi_{(P, Q, R)}, v_{(P, Q, R)}, \zeta_{(P, Q, R)}).
$$
From this we conclude that
$$\lambda_{(P, Q, R)}-\frac{\bru}{u}\in \frac{2\pi}{u} \integers $$
contradicting the non-arithmeticity condition \eqref{NA}.

Hence Theorem \ref{ThRenContr} is applicable giving parts (a) and (b) of the lemma. To prove part (c) note that
$$ \cR_n=\sum_{j\geq n} \pi_j A_j\dots A_{n+1} [\bu_n-l_n(\bu_n) v_n] . $$
Pick a small $\teps$ and split $\cR_n=\cR_n'+\cR_n''$ where
the first term contain the terms with
$j<n+\teps \ln t$ and the second term contain the terms with
$j\geq n+\teps \ln t.$ Choosing $\teps$ small enough we can ensure that $\cR_n'\leq \frac{t}{2}.$
On the other hand for terms in $\cR_n''$ we have $\theta^{n-j}\leq t^{-\teps |\ln\theta|}$ and hence
$\cR_n''<C t^{-\teps |\ln\theta|} w_n.$ Thus
$$\bP(\cR_n>t)\leq \bP\left(w_n> C t^{1+\teps|\ln \theta|}\right) $$
and so part (c) follows from part (b).
\end{proof}

\begin{proof}[Proof of Lemma \ref{LmARen}]
The result follows from $z^-$ part of Theorem \ref{ThRenContr}
applied to the same Markov process \eqref{OurMP} as in the proof of Lemma
\ref{LmWRen}.
\end{proof}

\section{Extensions.}
\label{ScExt}

Here we discuss some extensions of our results which are not used in the proof of Theorem \ref{ThMain} but are helpful in studying
other properties of the walk. The application of these results will be presented in a separate paper.

\subsection{Environment inside the trap.}
Fix $R\ge 1.$ Let $\mathbb{T}_j=[n_j-M_N, n_j]$ be the $j$-th trap. We call $\brn_j\in \mathbb{T}_j$ {\it $R$-center of $\mathbb{T}_j$} if
$\brn_j$ is the rightmost point in $\mathbb{T}_j$ such that $w_{\brn_j}>w_n /R$ for all $n\in \mathbb{T}_j.$ We choose $R$ so that for each $k$ we have
$$\bP(\lambda_n\dots \lambda_{n+k}\in \{R, R^{-1}\})=0.$$
In particular, if for each $k$ we have $\bP(\lambda_{n} \dots \lambda_{n+k}=1)=0$
then we can take $R=1$ so that $\brn_j$ will be the point with the maximal value of $w_{n}.$ Denote
$\omega^{(j)}=\tau^{\brn_j} \omega,$ where $\tau$ is the standard shift on the space of
environments. Theorem \ref{ThTraps} can be strengthened %straightened
in the following way.

\begin{theorem} \label{ThTrapsE}
 Assume that the non-arithmeticity condition \eqref{NA} holds. Then there exists
a probability measure $\tnu_\delta$ on $\Omega$ and
a constant $\bc$ such that the point process
\begin{equation}\label{PoissonE}
\left\{ \left(\frac{{\brn_j}}{N},\,\frac{\mm_{n_j}}{N^{1/s}}, \omega^{(j)}\right)
\right\}
\end{equation}
converges as $N\to\infty$ to a Poisson process on $[0,1]\times [\delta, \infty)\times \Omega$
with the measure $\bc dt' d\mu_\delta(\mm) d\nu_\delta(\omega).$
As $\delta\to 0$
$\mu_\delta$ converges to a measure with
density $\frac{\bc s}{\mm^{s+1}}$ and $\tnu_\delta$ converges to some measure $\tnu.$
\end{theorem}

In other words if the walker is trapped, then he sees the environment distributed according to a measure $\tnu.$
This statement extends the results obtained in \cite{Go, K2}.

To prove Theorem \ref{ThTrapsE}
we first show that as $N\to\infty$ and then $\delta\to 0$
\begin{equation}\label{PoissonEW}
\left\{ \left(\frac{{\brn_j}}{N},\,\frac{w_{\brn_j}}{N^{1/s}}, \omega^{(j)}\right)
\right\} \text{ converges to a Poisson process with measure }
\tc dt' \frac{dw}{w^{1+s}} d\tnu^*.
\end{equation}
The proof of this result is similar to the proof of Theorem \ref{ThTraps}. Namely, call $n$
{\it $(R, l)$--maximal} if $w_{n}>w_{n+k}/R$ for $0<k\leq l$ and for each $0<k'<l$
there exists $|k''|\leq l$ such that $w_{n+k'}\leq w_{n+k''}/R.$

Using the same argument as in Lemma \ref{LmLEss} one can show that
$$ \bP\left(\frac{w_n}{t}\in [\bra, \brra]\text{ and } n \text{ is } (R,l)-\text{maximal}|\zeta_n=\brzeta, \pi_n=\brpi, v_n=\brv
\right)
\approx \hf_l(\brzeta, \brpi, \brv) \left(\bra^{-s}-\brra^{-s}\right) t^{-s} $$
and $\hf_l\to \hf$ as $l\to\infty.$
Moreover similarly to Lemma \ref{LmFPos} one can show that $\bP(\hf>0)>0$
(otherwise we would get a contradiction with the fact that each trap has a center).
In addition we have that for each $\brg_{-l}\dots \brg_l$
$$ \bP\left(\frac{w_n}{t}\in [\bra, \brra]\text{ and } n \text{ is } (R,l)-\text{maximal}|\zeta_n=\brzeta, \pi_n=\brpi, v_n=\brv,
g_{n-l}=\brg_{-l}, \dots g_{n+l}=\brg_l
\right)$$
$$\approx \tf(\brzeta, \brpi, \brv, \brg_{-l}, \brg_l) \bP_{\zeta_{n+k}, \pi_{n+k}, v_{n+k}}
\left(w_{n+k} \in \frac{[\bra t, \brra t]}{\lambda_{n+1}\dots \lambda_{n+k}}\right) $$
This implies that if $\cB$ is $\cF_{-l, l}$ measurable then
$$ \bP\left(\frac{w_n}{t}\in [\bra, \brra], n \text{ is } (R,l)-\text{maximal} \text{ and } \tau^{n} \omega \in \cB\right)\approx
\tnu_l(\cB) \tc (\bra^{-s}-\brra^{-s}) t^{-s} $$
and $\tnu_l\Rightarrow\tnu$ as $l\to\infty.$
Now the proof of \eqref{PoissonEW} proceeds similarly to the proof of Theorem \ref{ThTraps}.
To pass from \eqref{PoissonEW} to \eqref{PoissonE} we note that
$\frac{\mm_j}{w_{\brn_j}}$ is well approximated by
$\sum_{|k|<l} l_{n+k}(\bu_{n+k}) \Lambda_{n,k}$ provided that $l$ is sufficiently
large. Here
$$ \Lambda_{n,k}=\begin{cases}
\lambda_{n+1}\dots \lambda_{n+k} & \text{if } k>0 \\
1 & \text{if } k=0 \\
\lambda_{n+k+1}\dots \lambda_{n} & \text{if } k<0
\end{cases} . $$
Accordingly, in the limit $l\to\infty,$ we have
$\mm_j=w_{\brn_j} \cH(\omega^{(j)})$ for some measurable function $\cH.$ Now Lemma \ref{LmPT} shows that
\eqref{PoissonEW} implies \eqref{PoissonE} with $d\tnu=\cH^s d\tnu^*.$

\subsection{Arithmetic case.}
We note that condition \eqref{NA} has been used in Section \ref{tail}
to show that \eqref{SpCBMat} does not have the solution. On the other hand
if \eqref{SpCBMat} has a non-trivial solutions then the analysis of
Appendix \ref{AppRenewal} has to be modified. Namely, the non-arithmetic
local limit theorem (Lemma \ref{LmLLT}) has to be replaced by its arithmetic version.
This will cause replacing $t^{-s}$ in the estimates of Theorem~\ref{ThRenContr}
by $t^{-s} g(\{\ln(t/\Delta)\})$
where $\{\dots\}$ denotes the fractional part,
$\Delta$ is the step of the progression containing the distribution of $\ln \lambda$
and $g$ is some continuous function. As a result the estimates of Section \ref{proofTh5}
have to be replaced by
$$ \bP(\rho_n>t)\sim t^{-s} g_1(\{\ln(t/\Delta)\}), \quad
\bP(\mm_j>t) \sim t^{-s} g_2(\{\ln(t/\Delta)\}). $$
Thus there would exists a measure $\mu$ on $\reals^+$ such that $\mu([t, \infty))=\brg(\{\ln(t/\Delta)\}) t^{-s}$
and the limit points of the distribution of the normalized hitting times will be of the form
$$ \sum_j \Theta_j (\Gamma_j-1) $$
where $\Gamma_j$ are iid mean 1 exponential random variables and $\Theta_j$ is a Poisson process with
measure $\mu_\brDelta$ for some $0\leq \brDelta<\Delta$ where
$\mu_\brDelta(A)=\mu(e^\brDelta A).$

In particular, we would like to note that regardless of condition \eqref{NA} we always have
\begin{equation}
\label{RhoAr}
\bP(\rho_n>t)\leq C t^{-s}.
\end{equation}

\appendix

\section{Occupation times for Markov chains.}
\label{AppOccupation}

We recall two facts about general Markov chains with discrete state space.
First, the number of visits to a given state conditioned on the event that
this state will be visited has a geometric distribution. Second,
consider a Markov chain with transition probabilities $p_{ij}.$
Let $\tp_{jk}$ be the probability that the chain starting at $j$ ever visits $k.$
Condition the chain on having at least one visit to $k.$ Then before coming to $k$ the
chain evolves as a Markov chain with transition probabilities
$$ p^*_{ij}=\frac{p_{ij} \tp_{jk}}{\sum_r p_{ir} \tp_{rk}}. $$
We now use this facts to analyze the joint distribution for the number of visits to
different sites.

Namely let $a$ and $b$ be two states of a transient chain such that
\begin{equation}
\label{HEl}
p_{ab}^{n_0}>\eps, \quad p_{ba}^{n_0}>\eps,
\end{equation}
where $p_{ab}^{n_0}$ denotes the transition probability after $n_0$ steps.

Let $q_a$ ($q_b$) denote the probability that $a$ (respectively $b$) is visited at least once and
$p_a$ ($p_b$) denote the probability that the chain started from $a$ (respectively $b$) does not return to that state
again. Let $\xi_a$ ($\xi_b$) be the number of visits to $a$ (respectively $b$).

\begin{lemma}\label{LemmaA1}
Given $n_0, \eps$ there exists a constant $C$ such that if
$q_a>\eps,$ $q_b>\eps$ then
$$ \Corr(\xi_a, \xi_b)>1-\frac{C}{\EXP(\xi_a)}. $$
\end{lemma}

\begin{proof}
We have $\xi_b=U+V+W$ where $U$ is the number of visits to $b$ before the
first visit to $a,$ $W$ is the number of visits to $b$ after the last visit
to $a$ and $V=\sum_{j=1}^n V_j$ where $V_j$ is the number of visits to $b$ between
$j$-th and $j+1$-st returns to $a.$ Then $V_j$ are iid. Let $v_{ab}=\EXP(V_j).$
Then by \eqref{HEl} we have the following uniform bounds
$$ v_{ab}=O(1), \quad \EXP(U)=O(1), \quad \EXP(W)=O(1). $$
Therefore
\begin{equation}
\EXP(\xi_b|\xi_a=k+1)=k v_{ab}+O(1).
\end{equation}
Hence
\begin{equation}
\label{EXPV}
\EXP(\xi_b)=v_{ab} \EXP(\xi_a)+O(1),
\end{equation}
$$ \EXP(\xi_a \xi_b)=v_{ab} \EXP(\xi_a^2)+O(\EXP(\xi_a)). $$
Combining the last two equalities we obtain
$$ \Cov(\xi_a, \xi_b)=v_{ab} \Var(\xi_a)+O(\EXP(\xi_a))$$
\begin{equation}
\label{CovVar1}
=v_{ab} \Var(\xi_a)\left(1+O\left(\frac{1}{\EXP(\xi_a)}\right)\right)
\end{equation}
where the last step uses the formulas
\begin{equation}\label{ExpVar}
 \EXP(\xi_a)=\frac{q_a}{p_a}, \quad \Var(\xi_a)=\frac{q_a(2-q_a-p_a)}{p_a^2}.
 \end{equation}
%fact that $\xi_a$ has geometric distribution.
Interchanging roles of $a$ and $b$ we get
\begin{equation}
\label{CovVar2}
\Cov(\xi_a, \xi_b)=v_{ba} \Var(\xi_b)\left(1+O\left(\frac{1}{\EXP(\xi_a)}\right)\right).
\end{equation}
In view of \eqref{EXPV} we have
$$v_{ab} v_{ba}=1+O\left(\frac{1}{\EXP(\xi_a)}\right)$$
so multiplying \eqref{CovVar1} and \eqref{CovVar2} we get
$$ \frac{\Cov^2(\xi_a, \xi_b)}{\Var(\xi_a) \Var(\xi_b)}=
1+O\left(\frac{1}{\EXP(\xi_a)}\right). $$
\end{proof}
Let $\Vert\dots \Vert$ denote the $L^2$ norm.

\begin{corollary}\label{CorA2}
Under the conditions of Lemma \ref{LemmaA1} there exists a constant $\brC$ (depending only on $\eps$ and $n_0$) such that
$$ \left\Vert\frac{\xi_a}{\EXP(\xi_a)}-\frac{\xi_b}{\EXP(\xi_b)}\right\Vert \leq \brC
\left[\frac{1}{\sqrt{\EXP(\xi_a)}}+(1-q_a)+(1-q_b)\right]. $$
\end{corollary}

\begin{proof}
We have
$$ \left\Vert\frac{\xi_a}{\EXP(\xi_a)}-\frac{\xi_b}{\EXP(\xi_b)}\right\Vert=
\left\Vert\frac{\xi_a-\EXP(\xi_a)}{\EXP(\xi_a)}-\frac{\xi_b-\EXP(\xi_b)}{\EXP(\xi_b)}\right\Vert$$
$$=
 \left\Vert\frac{\xi_a-\EXP(\xi_a)}{\sqrt{\Var(\xi_a)}}\frac{\sqrt{\Var(\xi_a)}}{\EXP(\xi_a)}-
\frac{\xi_b-\EXP(\xi_b)}{\sqrt{\Var(\xi_b)}}\frac{\sqrt{\Var(\xi_b)}}{\EXP(\xi_b)}
\right\Vert$$
$$\leq \left\Vert\frac{\xi_a-\EXP(\xi_a)}{\sqrt{\Var(\xi_a)}}-
\frac{\xi_b-\EXP(\xi_b)}{\sqrt{\Var(\xi_b)}}
\right\Vert+\left|\frac{\sqrt{\Var(\xi_a)}}{\EXP(\xi_a)}-1\right|+
\left|\frac{\sqrt{\Var(\xi_b)}}{\EXP(\xi_b)}-1\right| $$
Notice that the first term equals to $\sqrt{2(1-\Corr(\xi_a, \xi_b))}$ and so it can be estimated by Lemma \ref{LemmaA1}
while the last two terms can be estimated by \eqref{ExpVar}.
\end{proof}

To use the above corollary we need to estimate $1-q_a.$ This can be done using the following bound.

\begin{lemma}\label{LemmaA3}
Given $\eps, n_0$ there exists a constant $\tC$ such that if \eqref{HEl} holds then
$$ \Prob_b(\xi_a=0)\leq C p_b. $$
\end{lemma}

\begin{proof}
\eqref{HEl} implies that the probability to visit $a$ before the $n_0$-th return to $b$ is greater than $\eps.$
Accordingly
$$ \Prob(\xi_a=0|\xi_b=n)\leq (1-\eps)^{[n/n_0]}$$
where $[\dots]$ denotes the integer part. Hence
$$ \Prob_b(\xi_a=0)\leq \sum_{n=0}^\infty p_b (1-p_b)^n (1-\eps)^{[n/n_0]} $$
giving the required estimate.
\end{proof}

%Finally let $r_{ba}$ be the probability that the process started from $b$ visits $a.$

%\begin{lemma}
%Given $\eps$ there is a constant $C$ such that if $q_a>\eps,$ $q_b>\eps$ then
%$$ |\Cov(\xi_a, \xi_b)|\leq C r_{ba} \EXP(\xi_a)\EXP(\xi_b). $$
%\end{lemma}

%\begin{proof}
%We have
%$\xi_a=\tU+\tV$ where $\tU$ is the number of visits to $a$ before the first visit to $b$
%and $\tV$ is the number of visits to $a$ after the first visit to $b.$ Then $\tV
 %\end{proof}

\section{Poisson process and stable distributions.}
\label{SSPP}
Let $(X, \mu)$ be a measure space.
Recall that a Poisson process is a point process with values in $X$ such that
if $N(A)$ is the number of points in $A\subset X$
then $N(A_1), N(A_2)\dots N(A_k)$ are mutually independent
if $A_1, A_2\dots A_k$ are disjoint and
$N(A)$ has the Poisson distribution with parameter $\mu(A).$
If $X\subset \reals^d$ and $\mu$ has density $f$ with respect to the Lebesgue measure we
say that $f$ is the intensity of the Poisson process.
\begin{lemma}
\label{LmPT}
(a) If $\{\Theta_j\}$ is a Poisson process on $X$
and $\psi: X\to \tX$ is a measurable map
then $\tTheta_j=\psi(\Theta_j)$ is a Poisson process.
If $X=\tX=\reals$ and $\psi$ is invertible then the intensity of $\tTheta$ is
$$\tf(\theta)=f(\psi^{-1}(\theta)) \left(\frac{d\psi}{d\theta}\right)^{-1}. $$

(b) Let $(\Theta_j, \Gamma_j)$ be a point process on $X\times Z$ such that
$\{\Theta_j\}$ is a Poisson process on $X$
and $\{\Gamma_j\}$ are $Z$-valued random variables
which are i.i.d. and independent of $\{\Theta_k\}$
then $(\Theta_j, \Gamma_j)$ is a Poisson process on $X\times Z.$

(c) If in (b) $X=Z=\reals$ then
$\tTheta=\{\Gamma_j \Theta_j\}$ is a Poisson process. Its intensity is
$$ \tf(\theta)=\bE\left(f\left(\frac{\theta}{\Gamma}\right)\frac{1}{\Gamma}\right). $$
\end{lemma}

%\begin{lemma}
%\label{PChar}
%Let $\Theta$ be Poisson process on $X,$ $\psi:X\to\reals$ a measurable function and
%$$ V=\sum_j \psi(\theta_j). $$
%Then
%$$\bE(\exp(ivV))=\exp\left[\int\left(\exp(iv\psi(\theta)-1\right)d\mu(\theta)\right].$$
%Accordingly,
%\begin{equation}
%\label{PExp}
%\bE(V)=\int \psi(\theta) d\mu(\theta)
%\end{equation}
%and
%\begin{equation}
%\label{PVar}
%\Var(V)=\int \psi^2(\theta) d\mu(\theta)
%\end{equation}
%\end{lemma}

\begin{lemma}
\label{LmPSt}
(a) If $0<s<1$ and $\Theta_j$ is a Poisson process with intensity $\theta^{-(1+s)}$ then
$ \sum_j \Theta_j $ has stable distribution of index $s.$

(b) If $1<s<2$ and $\Theta_j$ is a Poisson process with intensity $\theta^{-(1+s)}$ then
$$ \lim_{\delta\to 0}
\left[\left(\sum_{\delta<\Theta_j} \Theta_j\right)-\frac{1}{(s-1)\delta^{s-1}} \right]$$
has stable distribution of index $s.$

(c) If $s=1$ and $\Theta_j$ is a Poisson process with intensity $\theta^{-2}$ then
$$ \lim_{\delta\to 0} \left[\left(\sum_{\delta<\Theta_j} \Theta_j\right)-|\ln\delta| \right]$$
has stable distribution of index $1.$
\end{lemma}

\section{Renewal theorem for a system of contractions.}\label{AppRenewal}
\subsection{Main result.}
Let $M_1$ and $M_2$ be compact metric spaces,
and  let $\omega=\{g_j\}$ be a sequence of iid $M_2$ valued random variables.
Suppose that there is a $C^\eta$ map $G: M_1\times M_2\to M_1.$
%such that for each $g\in M_2,$ the map
%$G(\cdot, g)$ is a contraction of $ M_1.$
Consider Markov process on $M_1$
$$\Phi_{j+1}=G_j(\Phi_j), \text{ where }G_j(\Phi)=G(\Phi, g_j).$$
We suppose that the maps $G_j$ are contractions in the sense that there exist
constants $C$ and $\theta<1$ such that if $\Phi_j'$ and $\Phi_j''$ are two
realizations
of this Markov chain starting from $\Phi_0'$ and $\Phi_0''$ and evolving in the same
environment $\omega$ then with probability 1 we have
\begin{equation}
\label{EventualContr}
d(\Phi_j', \Phi_j'')<C \theta^j.
\end{equation}
Denote $\tomega=\{(g_j, \Phi_j)\},$
$\Omega=(M_2)^\integers,$ $\tOmega=(M_1\times M_2)^\integers.$
Let $\Prob$ be the distribution of $\omega.$ In view of \eqref{EventualContr} there is
unique stationary $\tProb$ distribution of $\tomega$ whose projection onto
$\Omega$ is $\Prob.$
Namely, conditioned on $\omega$ the distribution of $\{\Phi_j\}$ is a delta measure
concentrated at $\{\brPhi_j\}$ where $\brPhi_j$ is constructed as follows.
Take $\Phi^*\in  M_1$ and let
$\Phi_{k,j}=G_{j-1} \dots G_{k+1} G_k (\Phi^*).$ Then
$\brPhi_j=\lim_{k\to-\infty} \Phi_{k,j}. $
Let $b(\Phi, g)$ be a positive $C^\eta$ function on $M_1\times M_2$ and
$\Delta(\tomega)$ be a positive continuous function on $\tOmega.$
Denote $b_j=b(\Phi_j, g_j),$ $\Delta_j=\Delta(\tau^j \omega)$ and $\tau$ denotes the shift.
For $k>0$ denote
$$z_{n, k}^+=\sum_{j=1}^{k} b_n b_{n+1} \dots b_{n+j} \Delta_{n+j} , \quad
z_{n, k}^-=\sum_{j=1}^{k} b_n b_{n-1} \dots b_{n-j} \Delta_j .
$$
Let
$$z_n^+=\lim_{k\to\infty} z_{n,k}^+, \quad z_n^-=\lim_{k\to\infty} z_{n,k}. $$
Denote $a=\ln b.$
Suppose that $\tEXP(a)<0$ but for any $t$ there exists $N$ such that
\begin{equation}
\label{NoNest}
 \tProb\left(\prod_{j=1}^N b_j>t\right)>0 .
\end{equation}

Denote by $\IS^1$ a set of complex numbers of absolute value 1.

\begin{theorem}
\label{ThRenContr}
Suppose that
for any numbers $u, \bru \in \reals,$ there exists no continuous function $\fh: M_1\to \IS^1$
such that the following equation is satisfied $\tProb$ almost surely
\begin{equation}
\label{SpCB}
e^{iu a(\Phi, g)}=e^{i\bru} \frac{\fh(\Phi)}{\fh(G(\Phi, g))}.
\end{equation}
Then there is are constant $s>0,$ $\brs>s$ and $\brC>0$ such that
such that

(a) If $k>\brC \ln t$ then
$$ \Prob(z_{n}^+-z_{n,k}^+>1)< \brC t^{-\brs}, \quad
\Prob(z_{n}^--z_{n,k}^->1)< \brC t^{-\brs}.
$$

(b) There exists a function $f(\Phi)$ such that
if $k>\brC t$ then
$$ \Prob(z_{n,k}^+>t| \Phi_0=\Phi)\sim f(\Phi) t^{-s} . $$
In particular
$$ \Prob(z_{n}^+>t| \Phi_0=\Phi)\sim f(\Phi) t^{-s} . $$

(c) There exists a measure $\tnu$ on $\tOmega$ such that for any continuous function $H$ on
$\tOmega$ the following asyptotics holds if $k >\brC \ln t.$
$$ \tEXP(1_{z_{n,k}^->t} H(\tomega))\sim t^{-s} \tnu(H). $$
In particular
$$ \tEXP(1_{z_{n}^->t} H(\tomega))\sim t^{-s} \tnu(H). $$
\end{theorem}

\subsection{Renewal theorem and large deviations.}
We will deduce Theorem \ref{ThRenContr} from a large deviation bound.
Let $y_n=\sum_{j=0}^{n-1} a_j.$
\begin{theorem}
\label{ThLD}
Suppose that \eqref{SpCB}
has no solutions. Then there is a number $\alpha^*>0$
and a strictly convex analytic function $\gamma: [0, \alpha^*)\to \reals$ such that
\begin{equation}
\label{InfDer}
\lim_{\alpha\to\alpha^*} \gamma'(\alpha)=+\infty
\end{equation}
and

(a) If $\alpha>\alpha^*$ then for each $\beta\in\reals,$ $\Prob_\Phi(y_n\geq \alpha n)=\cO(e^{-\beta n});$

(b) If $\tEXP(a)<\alpha<\alpha^*$ then for any $J\in [0, \infty]$ for any continuous function $H$ on~$\tOmega$
$$ \EXP_\Phi\left(1_{y_n-\alpha n\in J} H\left(\tau^n \tomega) \right)\right)\approx
\frac{e^{-\gamma(\alpha) n}}{\sqrt n} \nu_\alpha(H) h_\alpha(\Phi)
\int_J e^{-\gamma'(\alpha)t} \gamma'(\alpha) dt, $$

(c) If $\alpha<\alpha^*$ then for any $J\in [0, \infty]$ for any continuous functions $H, \hH$ on $\tOmega$
$$ \tEXP\left(1_{y_n-\alpha n\in J} \hH\left(\tomega\right)
H\left(\tau^n \tomega \right)\right)\approx
\frac{e^{-\gamma(\alpha) n}}{\sqrt n} \hnu_\alpha(\hH)\nu_\alpha(H)
\int_J e^{-\gamma'(\alpha)t} \gamma'(\alpha) dt.$$
\end{theorem}

Let $\alpha_0=\arg\min\frac{\gamma(\alpha)}{\alpha}$ and $s=\frac{\gamma(\alpha_0)}{\alpha_0}.$
Note that by \eqref{InfDer} the minimum is achieved strictly inside $(0, \alpha^*).$

Theorem \ref{ThLD} is proven in subsection \ref{SSLD}.
Here we use this theorem to obtain Theorem \ref{ThRenContr}.

\begin{lemma}
\label{LmMax}
There exist constants $C$ and $\ts>0$ such that

(a) for each $\Phi$ we have $\Prob_\Phi (\min y_{n}\geq Y)\leq C e^{-sY}; $

(b) For each $n$ we have $\Prob_\Phi(\exists l>k: y_{n}\geq Y$ and $y_{n+l}\geq Y)\leq \frac{C}{\sqrt{Y}} e^{-(sY+\ts k)}. $
In particular,
$\Prob_\Phi(\exists n_1, n_2: n_2>n_1+k $ and $y_{n_j}\geq Y$ for $j=1,2)\leq C e^{-(sY+\ts k)}. $
\end{lemma}

\begin{proof}
By Theorem \ref{ThLD}
$$ \Prob_\Phi(y_n\geq Y)\leq \frac{C}{\sqrt{n}} \exp\left[-\left(\frac{\gamma(Y/n)}{Y/n}\right)Y\right]$$
\begin{equation}
\label{LDGauss}
\leq\frac{C}{\sqrt{n}} \exp[-sY] \exp\left[-c\left(\frac{Y}{n}-\alpha_0\right)^2 Y\right].
\end{equation}
The main contribution to this sum comes from $n\approx Y/\alpha_0.$ For those $n$
$$ c\left(\frac{Y}{n}-\alpha_0\right)^2 Y \leq \tc\frac{(Y-\alpha_0 n)^2}{Y}. $$
Since
$$ \sum_n \frac{1}{\sqrt{Y}} \exp\left[-\tc\frac{(Y-\alpha_0 n)^2}{Y}\right]\leq C $$
proving part (a).

To prove (b) observe that by Markov property
$$\Prob_\Phi(y_{n}\geq Y \text{ and }y_{n+l}\geq Y)\leq \frac{C}{\sqrt{Y}} e^{-(sY+\ts k)}\leq
\Prob_\Phi(y_n\geq Y) \max_{\brPhi} \Prob_\brPhi(y_l>0) .$$
The second term is less than $\frac{C}{\sqrt{l}} e^{-\gamma(0) l}$ due to Theorem \ref{ThLD} while the first term is
less than $\frac{C}{\sqrt{Y}} e^{-sY}$ by part (a). Now the first inequality of part (b) follows by summation over
$l>k$ and the second one follows by summation over $n$ and $l.$
\end{proof}

\begin{lemma}
\label{Lm1stVisit}
Suppose that $n, Y\to\infty$ so that $\frac{n-\frac{Y}{\alpha_0}}{\sqrt{Y}}\to\beta$
Denote
$$\Omega_n=\{y_{n}\geq Y, y_{m}<Y \text{ for all } 0<m<n\}.$$
Then for each $\Phi$  and a continuous function $\hH: \reals\times \tOmega\to \reals$ the following limits exist
$$(a)\quad \lim_{n\to\infty} \Prob_\Phi (\Omega_n) \sqrt{Y} e^{sY};\Quad$$
$$(b)\quad \lim_{n\to\infty} \EXP_\Phi (1_{\Omega_n} \hH(y_n-Y, \tau^n \tomega ) \sqrt{Y} e^{sY}.\Quad$$
Moreover both limits are bounded by $\Const e^{-c\beta^2}.$
\end{lemma}
\begin{proof}
(a) By Lemma \ref{LmMax}(b)  it is enough to show that for each $k$ the following limit exists
\begin{equation}
\label{AlmMax}
 \lim_{n\to\infty} \Prob_\Phi (y_{n}\geq Y, y_{m}<Y \text{ for } n-k<m<n) \sqrt{Y}  e^{sY} .
\end{equation}
The limiting expression equals to
$$ \int \Prob_\Phi (y_{n-k} \in [Y-z_1, Y-z_2]) d P^\Phi_k(z_1, z_2) \sqrt{Y}  e^{sY} $$
where $P_k^\Phi$ is the distribution function of the random vector
$ (y_k, \max_{0<j\leq k} y_j) $ for our Markov chain started from $\Phi.$ The integral above is the limit of
Lebesgue-Stiltjes sums where each term has form
$$ \Prob_\Phi (y_{n-k} \in [Y-z_1, Y-z_2]) \Prob_{\Phi_{n-k}}(y_k\in [z_1, z_1+\eps), \max_{0<j\leq k} y_j\in [z_2, z_2+\eps))  \sqrt{Y}  e^{sY}. $$
Since for each $z_1, z_2$ the last probability is a function of $\Phi_{n-k}$ part (a) follows from Theorem \ref{ThLD}.

For part (b) it is sufficient to restrict our attention to a dense set of functions $\hH.$ In particular we can assume that
$\hH$ depends only on finitely many coordinates, that is we need to compute the limit
$$ \lim_{n\to\infty} \EXP_\Phi (1_{y_{n}\geq Y, y_{m}<Y \text{ for } n-k<m<n} \tH(\Phi_{n-k}, g_{n-k} \dots \Phi_{n+k}, g_{n+k})) \sqrt{Y}  e^{sY} .$$
for some $\tH: (M_1\times M_2)^{2k+1}\to\reals$ The analysis of the last limit is the same as the analysis of \eqref{AlmMax}.

 Finally the fact that above limits are $\cO\left(e^{-c\beta^2}\right)$ follows from the estimates in the
proof of Lemma \ref{LmMax} (see \eqref{LDGauss}).
\end{proof}

\begin{proof}
[Proof of Theorem \ref{ThRenContr}]
Let $Y=\ln t.$ Note that $z_n^+-z_{n,k}^+
=\sum_{j=k+1}^\infty \Delta_j e^{y_j}.$
Take small $\eps>0.$ Then if $t$ is sufficiently large, the inequality
$z_n^+-z_{n,k}^+\geq 1$ implies that there is $j>k$ such that
$y_j>(1-\eps)Y-\eps j.$ Note that $-\eps j<(1-\eps) Y-\eps j<\eps j$ provided $\brC$ is large enough.
Hence
$$\Prob_\Phi (y_j>(1-\eps)Y-\eps j)
\frac{C}{\sqrt{j}}  e^{-\brgamma j}$$
where $\brgamma=\min_{[-\eps, \eps]} \gamma.$
Summing over $j$ we get
$$\Prob_\Phi(z_n^+-z_{n,k}^+>1)\leq C(\eps) e^{-\brgamma k}. $$
Since $k\geq \brDelta \ln t$ and $\brDelta$ can be taken sufficiently large
this proves the first inequlity of part (a). The proof of the second inequlity is similar.

To prove part (b) take $M\gg 1.$
We claim that terms with $y_{-n}\leq Y-M$ can be ignored. Indeed for terms with
$Y-M-1<y_n<Y-M$ to make a contribution greater than $e^{-M/2}$ there should be at least $e^{M/2}/C$ such terms.
By Lemma \ref{LmMax} the probability of such an event is
$$\cO\left(\exp-[s(Y-M)+\ts \exp(M/2)]\right) $$
which establishes our claim. Therefore for large $M$ and $l$ we can approximate
$\Prob_\Phi(z_{n,k}^+\geq t) t^s $ and $\Prob_\Phi(z_{n}^+\geq t) t^s $ by
$\sum_n \Prob_\Phi (1_{\Omega_{n,M}} 1_{B_{n,l, \eps}}) t^s$
where
$$\Omega_{n,M}=\{y_{-n}\geq Y-M, y_{-m}<Y-M \text{ for } 0\leq m<n\}, $$
$$ B_{n,l}=\left\{\sum_{|m-n|<l} \Delta_{m} e^{y_{m}} \geq t-\eps \right\}
. $$
The fact that the last
$\sum_n \Prob_\Phi (1_{\Omega_{n,M}} 1_{B_{n,l, \eps}}) t^s$
approaches the limit as $t\to\infty$ follows from Lemma \ref{Lm1stVisit}(b).
This proves part (b). The proof of part (c) is similar.
\end{proof}

\subsection{Large deviations.}
\label{SSLD}
We follow the approach of \cite{W}.

To simplify the notation we assume for the rest of this section that \eqref{EventualContr} holds with $C=1,$
the general case can be reduced to this one by considering our Markov chains only at the times which are multiples of a
sufficiently large $n_0.$

Consider operators $P_\kappa$ given by
$$ P_\kappa(h)(\Phi)=\EXP_\Phi(e^{\kappa a(\Phi, g)} h(\Phi_1)). $$
$P_\kappa$ is a positive operator
preserving the space of $C^\eta$ functions. Moreover it has many invariant cones as we describe below.
Let
$$ \cC_K=\{h\geq 0: \text{ for all }\tPhi, \ttPhi \text{ we have } h(\tPhi)\leq e^{K d^\eta(\tPhi, \ttPhi)} h(\ttPhi) \}. $$
A direct computation (using \eqref{EventualContr} with $C=1$) shows that
$P_\kappa(\cC_K)\subset \cC_\brK$ where $\brK=K \theta^\eta+\kappa \bH$ and
$\bH$ is the Holder constant of $a$ with respect to $\Phi$ variable.   Now \cite{L} shows that if $K$ is so large that
$K>\brK$ then $P_\kappa$ contracts the Hilbert metric on $\cC_K$ and so there exist
positive eigenfunctions
\begin{equation}
\label{PosEigen}
 P_\kappa h_\kappa=e^{\lambda_\kappa} h_\kappa
\end{equation}
in $\cC_K$ and moreover for any two elements $h', h''$ of $\cC_K$ the directions of
$P_\kappa^n h'$  and $P_\kappa^n h''$ converge to each other exponentially fast.
This in turn implies that the rest of the spectrum of $P_\kappa$
is contained in a disc of radius strictly smaller than
$e^{\lambda_\kappa}$. Since $e^{\lambda_\kappa}$ is an isolated eigenvalue of $P_\kappa,$
$\lambda_\kappa$ depends analytically on $\kappa.$

%\begin{remark}
%The fact that the leading eigenvalue of $P_\kappa$ is isolated implies that there is $\eps>0$ such that
%$$ P_\kappa^n h=e^{\lambda_\kappa n} \left(\nu_\kappa(h) h_\kappa+O(e^{-\eps n}\right) . $$
%For the Markov process we are interested in (see Section \ref{tail}) this statement reduces to part (d) of Lemma \ref{Mainlemma}.
%\end{remark}

%Next
%$$ P_\kappa^n(h)(\omega)=\EXP_\Phi(e^{\kappa y_n} h(\Phi_n)) $$
%Since $P^n_\kappa h_\kappa=e^{n \lambda_\kappa} h_\kappa$ it follows
%that
%$$ \Vert P_\kappa^n h \Vert_{C^0} \leq C e^{n\lambda_\kappa} $$
%where $C=\frac{\max h_\xi}{\min h_\xi}.$
%Next since $\Phi_j$ are contractions in the sense of \eqref{EventualContr}
%implies
%that the following Doeblin-Fortet inequality holds
%\begin{equation}
%\label{DF1}
%\Vert P_\kappa^n h \Vert_{C^\eta}\leq e^{n \lambda_\kappa}
%\left[C_1 \Vert h \Vert_{C^0}+C_2 \Vert h \Vert_{C^\eta} \theta^n \right].
%\end{equation}
%Now the Ruelle-Perron-Frobenius Theorem (see e.g. \cite{L})
%implies the following
%$e^{\lambda_\kappa}$ is an isolated eigenvalue
%which, therefore, depends analytically on $\xi;$
We need the fact that the map $\kappa\to\lambda_\kappa$ is strictly convex.
To see this we need formulas for derivatives of $\lambda$ with respect to $\kappa.$
To this end let
$\nu_\kappa$ be the eigenvector of the adjoint
operator
$$ \nu_\kappa(P_\kappa h)=e^{\lambda_\kappa} \nu_\kappa(h).$$
Differentiating \eqref{PosEigen} we get
$$ \EXP_\Phi\left(a(\Phi, g) e^{\kappa a(\Phi, g)} h_\kappa(G(\Phi, g))\right)+P_\kappa(h_\kappa')=
\lambda' e^{\lambda_\kappa} h_\kappa+e^{\lambda_\kappa} h_\kappa'.$$
Applying $\nu_\kappa$ to both sides we get
\begin{equation}
\label{DerLambda}
\lambda_\kappa=\frac{\nu_\kappa(\EXP(a(\Phi, g) e^{\kappa a(\Phi, g)} h_\kappa(G(\Phi, g))))}
{\nu_\kappa(e^{\lambda_\kappa} h_\kappa)}.
\end{equation}
Let
$$\ta_\kappa=\kappa a-\lambda_\kappa+\ln h_\kappa-\ln h_\kappa(G(\Phi, g)) . $$
Then
\begin{equation}
\label{EqNormalized}
\EXP_\Phi\left(e^{\ta_\kappa}\right)=1
\end{equation}
so we can consider a Markov chain with generator
$$ \tP_\kappa(h)(\Phi)=\EXP_\Phi\left(e^{\ta_\kappa} h(\Phi_1) \right). $$
Observe that
$$ \tP_\kappa=e^{-\lambda_\kappa} M_\kappa^{-1} P_\kappa M_\kappa $$
where $M_\kappa$ denotes the multiplication by $A_\kappa$
so the eigenvalue of the adjoint operator
(which is the stationary measure for our
Markov process) equals
$$m_\kappa(h)=\nu_\kappa(h h_\kappa).$$
Normalize $m_\kappa$ by the condition $m_\kappa(1)=1.$
Then $m_\kappa$ is the invariant measure for the Markov process with transition operator
$P_\kappa.$ Denoting by $\bm_\kappa$ the corresponding invariant measure on $\tOmega$
we can rewrite \eqref{DerLambda} as
\begin{equation}
\label{DerAver}
\lambda_\kappa'=\bm_\kappa(a).
\end{equation}
Next we compute $\lambda_\kappa''.$ Fix a $\kappa_0$ and let
$$\brP_\kappa h=\EXP_\Phi(e^{\bra_\kappa} h(\Phi_1)) $$
where
$$ \bra_\kappa=\kappa(a-\bm_{\kappa_0}(a))+\ln h_{\kappa_0}-\ln h_{\kappa_0}(G(\Phi, g))-\lambda_{\kappa_0}+
\kappa_0 \bm_{\kappa_0}(a) . $$
Then the leading eigenvalue of $\brP_\kappa$ is
$$ \brlambda_\kappa=\lambda_ \kappa-(\kappa-\kappa_0) \bm_{\kappa_0}(a)-\lambda_{\kappa_0} $$
and so $\brlambda''(\kappa_0)=\lambda''(\kappa_0). $
Let $\brh_\kappa$ be the leading eigenvector of $\brP_\kappa$ and $\brmu_k$ be the leading eigenvalue of
the adjoint operator. Then we have
$$ \brP_\kappa'' h+2 \brP_\kappa' \brh_\kappa'+ P_\kappa \brh_\kappa''=
e^{\brlambda_\kappa} (\brlambda_\kappa')' \brh_\kappa+e^{\brlambda_\kappa} \brlambda_\kappa'' \brh_\kappa+
2 e^{\brlambda_\kappa} \brlambda_\kappa' \brh_\kappa'+e^{\brlambda_\kappa} \brh_\kappa''. $$
Applying $\brnu_\kappa$ to both sides and using that
$$ \brlambda_{\kappa_0}=0, \brlambda_{\kappa_0}'=0, \brh_{\kappa_0}=1, \brnu_{\kappa_0}=m_{\kappa_0} $$
we get
$$ \lambda_{\kappa_0}''=\bm_{\kappa_0}(\ha^2)+2\bm_{\kappa_0} (\ha(\Phi_0, g_0) \brh_{\kappa_0}(\Phi_1)) $$
where $\ha=a-\bm_{\kappa_0}(a). $
Applying the same argument to $P_\kappa^n$, which has the leading eigenvalue $e^{n\lambda_\kappa}$ we get
$$ n\lambda_{\kappa_0}''=\bm_{\kappa_0}
\left(\left(\sum_{j=0}^{n-1}\ha (\Phi_j, g_j) \right)^2\right)+2\bm_{\kappa_0}
\left(\left(\sum_{j=0}^{n-1} \ha(\Phi_j, g_j) \right) \brh_{\kappa_0}(\Phi_n)\right) . $$
Since the Markov process with transition operator $\tP_\kappa$ has a spectral gap the measure $\bm_\kappa$ is ergodic
and hence
$$ \frac{1}{n} \sum_{j=0}^{n-1} \ha(\Phi_j, g_j) \to 0 $$
almost surely. Therefore
\begin{equation}
\label{SecDer}
\lambda_{\kappa_0}''=\lim_{n\to\infty} \frac{1}{n} \bm_{\kappa_0}
\left(\left(\sum_{j=0}^{n-1}\ha (\Phi_j, g_j) \right)^2\right) .
\end{equation}
Since the RHS of the last expression is  positive we conclude that
$\lambda_\kappa$ is convex. We now show following the argument of Theorem 12 of \cite{BG2} that
$\lambda_\kappa$ is actually strictly convex.
Consider the following operator on $C^\eta(M_1\times M_2)$
$$ (\hP_\kappa h)(\Phi_0, g_0)=\tEXP^\kappa_{\Phi_0, g_0} h(\Phi_1, g_1). $$
Denote $\Gamma=(1-\hP_\kappa)^{-1} \ha=\sum_{j=0}^\infty \hP_\kappa^j.$
Then a direct computation shows that the RHS of \eqref{SecDer} equals to
$ \bm_\kappa(\Gamma^2-(\hP_\kappa \Gamma)^2).$ Hence if $\lambda_{\kappa_0}''=0$ then
we have
$$ \bm_\kappa((\hP_{\kappa_0}\Gamma)^2)=\bm_{\kappa_0}(\Gamma^2). $$
Since $\bm_\kappa$ is stationary for $\hP_\kappa$ this implies that
$$ \bm_{\kappa_0}((\hP_{\kappa_0}\Gamma)^2)=\bm_{\kappa_0}(\hP_{\kappa_0}(\Gamma^2)). $$
Now Jensen inequality tells us that
$\Gamma(\Phi, g)$ is actually independent of $g,$ $\Gamma=\Gamma(\Phi).$
Then
$$\ha(\Phi, g)=\Gamma(\Phi)-(\hP_{\kappa_0} \Gamma)(\Phi, g)=
\Gamma(\Phi)-\Gamma(\Phi g) $$
contradicting \eqref{NoNest} (as  well as \eqref{SpCB}).
This proves that $\lambda_\kappa$ is strictly convex.

Let $\alpha^*=\lim_{\kappa\to+\infty} \lambda'_\kappa.$ This limit exists since $\lambda_\kappa'$ is increasing and
is finite since $\lambda_\kappa\leq \kappa ||a||_{C^0}. $
We now prove Theorem \ref{ThLD} with this value of $\alpha^*.$

\begin{proof}
To prove part (a) we
iterate \eqref{EqNormalized} to get
$$ \EXP_\Phi\left(e^{\ty_n}\right)=1 $$
where
$$ \ty_n=\xi y_n-n \lambda_\kappa+\ln h(\Phi_0)-\ln h(\Phi_n). $$
Hence by Markov inequality
$$ \Prob_\Phi(y_n>n\alpha)\leq C e^{-n(\kappa \alpha-\lambda_\kappa)}\leq
C e^{n\kappa(\alpha^*-\alpha)} $$
where the last inequality uses that $\lambda_\kappa< \kappa \alpha^*.$ This proves part (a).

To prove part (b) suppose that $\kappa$ is such that $m_\kappa(a)=\alpha.$
Let $\tEXP_\Phi^\kappa$ denote the expectation with respect to the Markov
process with
generator $\tP_\kappa.$
\begin{lemma}
\label{LmLLT}
If \eqref{SpCB} has no solutions then there exists a function $\phi(\Phi)$ such that
\begin{equation}
\label{EqLLT}
\tEXP_\Phi^\kappa\left(1_{y_n-n\alpha\in I} H(\tau^n \tomega)\right)\sqrt{n}
\to \Leb(I) \phi(\Phi) \nu(H).
\end{equation}
\end{lemma}
The proof of this Lemma is given in subsection \ref{AppSSLLT}.

Now take $I=[t, t+\eps]$ then the RHS equals $\eps\phi(\Phi) \nu(H)$
while the LHS equals
$$ \sqrt{n} h_\kappa (\Phi) \EXP_\Phi\left(e^{n(\alpha\kappa-\lambda_\kappa)} \frac{H(\tau^n \tomega)}{h_\kappa (\Phi_n)} e^{\kappa t} \right) (1+o_{\eps\to 0}(1)). $$
Dividing $J$ into the segments of length $\eps\ll 1$ we obtain part (b). Part (c) follows from part (b) and the Markov property.
Finally observe that
$$ \gamma(\alpha)=\alpha\kappa-\lambda_\kappa $$
where $\kappa$ satisfies $\lambda_\kappa'=\alpha.$
Thus
$$ \frac{\partial\gamma}{\partial\alpha}=\left(\alpha-\lambda_\kappa'\right)\frac{\partial\kappa}{\partial\alpha}+\kappa=\kappa.
$$
This proves \eqref{InfDer}.
\end{proof}

\subsection{Local Limit Theorem.}
\label{AppSSLLT}
Consider
$$ \tP_{\kappa, u}(h)(\Phi)=\EXP_\Phi\left(e^{\ta_\kappa-iu(a_\kappa-\alpha)}
h(\Phi_1)\right).$$
Then
$$ \tP^n_{\kappa, u}(h)(\Phi)=\EXP_\Phi\left(\exp\left[\sum_{j=0}^{n-1} {\ta_\kappa-iu(a_\kappa-\alpha)}(\Phi_j, g_j)\right]
h(\Phi_n)\right).$$
Since
$$ \left|\sum_{j=0}^{n-1} \left[(\ta_\kappa-iu(a_\kappa-\alpha))(\tPhi_j, g_j)-(\ta_\kappa-iu(a_\kappa-\alpha))(\ttPhi_j, g_j)\right]\right|
\leq C \sum_{j=0}^{n-1} d(\tPhi_j, \ttPhi_j) \leq \tC d(\tPhi_0, \ttPhi_0) $$
we get
\begin{equation}
\label{DF2}
 \Vert P_{\kappa,u}^n h \Vert_{C^\eta}\leq
\left[\Vert h \Vert_{C^0}+C(u) \theta^n \Vert h \Vert_{C^\eta}\right].
\end{equation}
Using the theory of Doeblin-Fortet operators (\cite{L}) we conclude that for each $\theta'>\theta$ the spectrum of
$P_{\kappa, u}$ outside the disc of radius $\theta'$ consists of a finite number of
 eigenvalues with absolute values
at most 1. We claim that in fact there are no eigenvalues of absolute value 1.
Indeed let $e^{i\bru}$ be such an
eigenvalue and $\fh$ be the corresponding eigenfunction. Then
\begin{equation}
\label{EigVect}
\EXP_\Phi(e^{\ta_\kappa+iu a(\Phi)} \fh(\Phi_1))=e^{i\bru} \fh(\Phi).
\end{equation}
Let $\Phi^*=\arg\max|\fh|.$ Without loss of generality we can assume that
$|\fh(\Phi^*)|=1.$ Now \eqref{EqNormalized} implies that \eqref{EigVect} is only possible if
$|\fh(G(\Phi^*))|=1$ with probability 1.
Iterating we see that for all $n$
\begin{equation}
\label{ZetaOne1}
 |\fh(G_n\dots G_1 (\Phi^*))|=1.
\end{equation}
We claim that this implies that
\begin{equation}
\label{ZetaOne2}
|\fh(\Phi)|\equiv 1
\end{equation}
on the support of $\tmu.$ Indeed if $|\fh|<1-\eps$ on a
relatively open subset $U$ of $\supp(\tmu)$ then there would exist
$\brPhi \in M_1$ and $n_k\to\infty$ and such that
$G_{n_k} \dots G_1 (\brPhi) \in U$ with positive probability. Since $G_n\dots G_1$ contracts with speed
$\theta^n$ for large $k$  we would have
$|\fh(G_{n_k}\dots G_1 (\Phi^*))|<1-\eps/2$ with positive probability, contradiciting \eqref{ZetaOne1}.
Now \eqref{ZetaOne2} and \eqref{EqNormalized} show that
$$ e^{iu a(\Phi, g)} \fh(G(\Phi, g))=e^{i\bru} \fh(\Phi) $$
which contradicts \eqref{SpCB}.

We are now ready to prove Lemma \ref{LmLLT}. Since the LHS of \eqref{EqLLT} is monotone function of $H$ it suffices
to prove the result for a dense set of functions. In particular we may assume that $H$ depends only on finitely many coordinates
$$H=\tH(\omega_{-k+1}, \dots, \omega_0, \dots \omega_{k-1}) . $$
Then
\begin{equation*}
%\label{LLTContLoc}
 \sqrt{n}\tEXP_\Phi^\kappa\left(1_{y_n-n\alpha\in I} H(\tau^n \tomega)\right)=
\sqrt{n} \int
\tEXP_\Phi^\kappa\left(1_{y_{n-k}-(n-k)\alpha-Z-k\alpha \in I} h(\Phi_{n-k}, Z)\right)
d P_k(Z)
\end{equation*}
where
$h(\Phi, Z)=\EXP_{\Phi_0=\Phi}(\tH(\tau^{-k}\omega)|y_k=Z)$
and $P_k(z)$ is the distribution function of $y_k.$ Observe that for each $Z$ the RHS has the same form as the LHS of
\eqref{EqLLT} except that $n$ is replaced to $n-k$ and $h$ depends only on one coordinate.
Therefore it suffices to prove \eqref{EqLLT} in the case where $H=h(\Phi_0).$

Let $\Gamma_\theta(y)=\frac{1}{\pi} \frac{1-\cos(\delta y)}{\delta y^2} e^{i \theta y}.$ Then
$\hGamma_0(u)=(1-\frac{|u|}{\delta})_+$ and $\hGamma_\theta(u)=\hGamma_0(u+\theta).$
By Section 2.5 of \cite{Du} it suffices to show that for each $\theta, \delta$ we have
\begin{equation}
\label{ConvCutOff}
\sqrt{n} \tEXP_\Phi^\kappa (\Gamma_\theta(y_n-\alpha n) h(\Phi_n))\to \phi(\omega) \nu(h)
\int \Gamma_\theta(y) dy.
\end{equation}
We have the following inversion formula
$$ \tEXP_\Phi^\kappa (\Gamma_\theta(y_n-\alpha n) h(\Phi_n))=\frac{1}{2\pi}
\int_{-M}^M \hGamma_\theta(u) \tP_{\kappa, u}^n(h) du $$
where $M$ is such that $\hGamma_\theta(u)=0$ outside $[-M, M].$

Next, take a small $\eps_0.$ Then for $|u|<\eps_0$ we have the decomposition
$$ \tP_{\kappa, u}(h)=\lambda_{\kappa, u} m_{\kappa, u}(h) h_{\kappa, u}+\cR_{\kappa, u} $$
where
$$ \cR_{\kappa, u}(h_{\kappa, u})=0, \quad m_{\kappa, u}(R_{\kappa, u} h)=0 \text{ and }
\Vert\cR_{\kappa, u}^n\Vert\leq K\brtheta^n \text{ for some }\brtheta<1. $$
It follows that
$$ \sqrt{n} \int_{|u|<\eps_0} \hGamma_\theta(u) \tP_{\kappa, u}^n (h) du=
\frac{\sqrt{n}}{2\pi} \int_{-\eps_0}^{\eps_0} \hGamma_\theta(u) \lambda_{\kappa, u}^n m_{\kappa, u}(h) \
h_{\kappa, u} du +\cO(\brtheta^n). $$
Next, letting $u=\frac{t}{\sqrt n}$ we can rewrite above integral as
\begin{equation}
\label{IntEig}
 \frac{1}{2\pi} \int_{-\eps_0 \sqrt n}^{\eps_0 \sqrt n} \Gamma(t/\sqrt{n}) \lambda_{\kappa, t/\sqrt n}^n
m_{\kappa, t/\sqrt{n}}(h) h_{\kappa, t/\sqrt{n}} dt .
\end{equation}
The computations of the previous section give $\lambda_{\kappa, 0}=1, \lambda'_{\kappa, 0}=0$ so that
$\lambda_{\kappa, t/\sqrt n}^n\to e^{-\frac{\sigma^2 t^2}{2}},$ where $\sigma^2=\lambda_\kappa'',$
and this convergence is dominated,
that is, if $\eps_0$ is small enough and $|t|\leq\eps_0\sqrt{n}$ we have $\lambda_{\kappa, t/\sqrt n}^n\leq e^{-\frac{\sigma^2 t^2}{4}}.$
As $u\to 0$ we have $m_{\kappa, u}\to m_\kappa,$ $h_u\to 1$ so that the integral
\eqref{IntEig} converges
to $\sigma \hGamma_\theta(0) m_\kappa(h). $
On the other hand since for $\eps_0\leq |u|\leq M$ the spectral radius of $\tP_{\kappa, u}$
is strictly less than 1
we have
$$ \frac{\sqrt{n}}{2\pi} \int_{\eps_0<|u|<M} \hGamma_\theta(u) \tP_{\kappa, u}^n (h) du=
\cO(\sqrt{n} \brtheta^n)  $$
and \eqref{ConvCutOff} follows.

\section{Contracting property of the Markov chain~$\psi_n.$}
\label{ScIndBC}
For two stochastic matrices $\psi_{0}'$ and $\psi_{0}''$ define $\psi_{n}'$ and $\psi_{n}''$
for $n\ge 1$ using the second formula in \eqref{EqPsi}. The sequence  $\{P_k, Q_k, R_k\}_{1\le k\leq n}$
used in \eqref{EqPsi} is in both cases the same one.

Our goal is to estimate the norm $||\psi_n''-\psi_n'||$. We shall prove the following
\begin{proposition}\label{propD1}
Assume that condition \eqref{EqC2*}
holds. Then there are constants
$K=K(\eps)>0, \theta=\theta(\eps)<1$ such that
$$ ||\psi_{n}''-\psi_{n}'||\leq K \theta^n ||\psi_0''-\psi_0'||. $$
\end{proposition}
\begin{proof}
Consider two walkers $X'(t)$ and $X''(t)$ starting at the same site $(n,i)$ in the layer
$L_n$ and moving on $[L_0,L_{n+1}]$ with reflecting boundary conditions
$\psi_0'$ and $\psi_0''$ at $L_0$ respectively and absorbing boundary conditions at $L_{n+1}$.

We shall show that there exists a coupling between the walkers such that
\begin{equation}\label{D1}
\mathbb{P}(X'(\tT_{n+1}')\neq X''(\tT_{n+1}''))\leq K \theta^n ||\psi_0''-\psi_0'||,
\end{equation}
where $\tT_{n+1}'$ and $\tT_{n+1}''$ are the hitting times of $L_{n+1}$ for $X'$ and $X''$
respectively. The statement then follows since, according to the probabilistic meaning
of $\psi_n'$ and $\psi_n''$,
$$\psi_n'(i,j)=\mathbb{P}(X'(\tT_{n+1}')=j), \ \ \psi_n''(i,j)=\mathbb{P}(X''(\tT_{n+1}'')=j)$$
and therefore
$$\sum_{j=1}^m|\psi_n'(i,j)-\psi_n''(i,j)|\le 2\mathbb{P}(X'(\tT_{n+1}')\neq X''(\tT_{n+1}'')).$$

The coupling is constructed as follows.

1. The walkers walk together until they either reach $L_0$ for the first time or reach $L_{n+1}$
without visiting $L_0$ (after that they stop). Note that the trajectories which miss $L_0$ do not
contribute to the left hand side of \eqref{D1}.

2. Let $t_0$ be the first time the walkers reach $L_0$ and $t_1'-1$, $t_1''-1$ be the last
time $X'$ and $X''$ reach $L_0$ before reaching $L_{n+1}.$
Between $t_0$ and $t_1'$ and between $t_0$ and $t_1''$ the walkers move independently.

%3. $X'$ chooses its trajectory $Traj(X'):=(X'(t))_{ t_1\le t\le T_{n+1}'}$ (according to the
%conditional distribution distribution defined above).

%4. $X''$ waits until $X'$ reaches $L_{n+1}$ and then starts moving independently of $X'$
%until it meets a point belonging to trajectory while
Note that after time $t_1$ the walkers again move in the same environment
$\{P_k, Q_k, R_k\}_{1\le k\leq n}$ but conditioned on reaching $L_{n+1}$ before $L_0.$

3. If $X'(t_1(X'))=X''(t_1(X'')$ then they move together until they reach $L_{n+1}$.

4. Denote by $t_l',\ t_l''$ the hitting times of $L_l$ for $X'$ and $X''$, where $l\ge 2$.
If $X'(t_1(X'))\neq X''(t_1(X''))$ then they continue to move independently until
they reach $L_k$ with the minimal $k$ such that $X'(t_k')=X''(t_k'')$.

Note that the only trajectories that contribute to the left hand side of \eqref{D1}
are those for which $X'(t_k')\neq X''(t_k'')$ for all $k\in[2,n+1].$

We shall use on the following estimate
\begin{lemma}
\label{Lm1SExit} Let $X'(t_1')=(1, J'),$ $X''(t_1'')=(1, J'').$
There exists a constant $C$ such that
$$ d(J', J'')\leq C ||\psi_0''-\psi_0'||$$
where $d$ denotes the variational distance between the corresponding
distributions.
\end{lemma}
%Let us explain how Lemma \ref{Lm1SExit} allows us to finish the proof of the Proposition.
If at step $k$ of the procedure described above the walkers are
uncoupled then condition \eqref{EqC2*} guarantees that the probability
that they become coupled at $L_{k+1}$ is at least $\eps m.$
Thus the probability that the walkers were uncoupled at time $t_1$ and remain
uncoupled until step $n+1$ is less than
$$ C d(J', J'') \theta^n$$
and the result follows from Lemma \ref{Lm1SExit}.
\end{proof}

\begin{proof}[Proof of Lemma \ref{Lm1SExit}]
Denote
$$B(j, k)=\mathbb{P}(X(t_1)=(1,k)|X(t_1 +1)=(1,j)),
\quad a(j)=\mathbb{P}(\tT_{n+1} < \tT_1|X(t_0+1)=(1,j)), $$
$$\Gamma(j,k)=\bP(X \text{ returns to (1,k) after visiting } L_2
\text{ but before } \tT_{n+1}|X(\brt+1)=(1,j))$$
Also note that due to ellipticity we have
$$ \min_j a_j\geq \eps \max_j a_j. $$
Let $M$ be the matrix with $M(j,k)=a(j)$. %on diagonal and zeroes elsewhere.

With this notation our goal is to establish Lipshitz dependence of
$B$ on the boundary condition $\psi.$ We have
$$ B=M+Q \psi B+R B+\Gamma B $$
that is
$$ B=(I-Q\psi-R-\Gamma)^{-1} M. $$
Therefore if $B'$ and $B''$ correspond to different boundary conditions
$\psi'$ and $\psi''$ respectively then
$$ B'-B''= (I-Q\psi'-R-\Gamma)^{-1}
Q(\psi'-\psi'') (I-Q\psi''-R-\Gamma)^{-1} M.$$
The estimate we need is a consequence of two inequalities below.
\begin{equation}
\label{1StepIE1}
|| (\psi'-\psi'') (I-Q\psi''-R-\Gamma)^{-1} ||\leq C ||\psi'-\psi''||
\end{equation}
\begin{equation}
\label{1StepIE2}
|| (I-Q\psi'-R-\Gamma)^{-1} ||\leq C ||M||^{-1}
\end{equation}
To prove \eqref{1StepIE1} let $U=Q\psi''+R+\Gamma.$ Given a probability vector
$p$ let $\pi_k'=p\psi' U^k,$ $\pi_k''=p\psi''U^k.$ Due to ellipticity
$$ ||\pi_k'-\pi''_k||\leq C (1-\eps)^k ||\pi_0'-\pi_0''||\leq
\tC (1-\eps)^k ||\psi'-\psi''||$$
proving \eqref{1StepIE1}.

To prove \eqref{1StepIE2} let $U=Q\psi'+R+\Gamma.$ We have
$$ \sum_k U(j,k)=1-a_j.$$
Hence
$$ ||U||=\max_j \sum_k U(j,k)=1-\min_j a_j\leq 1-\eps \max_j a_j=
1-\eps ||M||. $$
Accordingly $||U^k||\leq (1-\eps ||M||)^k $ so that
$||(1-U)^{-1}||\leq (\eps||M||)^{-1}$ proving \eqref{1StepIE2}.
\end{proof}

\end{document}